\documentclass{article}%
\usepackage{amsmath}
\usepackage{amsfonts}
\usepackage{amssymb}
\usepackage{graphicx}
\usepackage[a4paper]{geometry}%
\providecommand{\U}[1]{\protect\rule{.1in}{.1in}}
\newtheorem{theorem}{Theorem}

\newtheorem{corollary}[theorem]{Corollary}

\newtheorem{lemma}[theorem]{Lemma}

\newtheorem{proposition}[theorem]{Proposition}

\newenvironment{proof}[1][Proof]{\noindent\textbf{#1.} }{\ \rule{0.5em}{0.5em}}
\begin{document}

\title{High-frequency propagation for the Schr\"{o}dinger equation on the torus}
\author{Fabricio Maci\`{a}\thanks{This research has been supported by grants
MTM2007-61755 (MEC) and Santander-Complutense 34/07-15844.}\\Universidad Polit\'{e}cnica de Madrid \\DEBIN, ETSI Navales\\Avda. Arco de la Victoria s/n. 28040 Madrid, Spain\\E-mail address: \texttt{fabricio.macia@upm.es}}
\date{}
\maketitle

\begin{abstract}
The main objective of this paper is understanding the propagation laws obeyed
by high-frequency limits of Wigner distributions associated to solutions to
the Schr\"{o}dinger equation on the standard $d$-dimensional torus
$\mathbb{T}^{d}$. From the point of view of semiclassical analysis, our
setting corresponds to performing the semiclassical limit at times of order
$1/h$, as the characteristic wave-length $h$ of the initial data tends to
zero. It turns out that, in spite that for fixed $h$ every Wigner distribution
satisfies a Liouville equation, their limits are no longer uniquely determined
by those of the Wigner distributions of the initial data. We characterize them
in terms of a new object, the \emph{resonant Wigner distribution}, which
describes high-frequency effects associated to the fraction of the energy of
the sequence of initial data that concentrates around the set of resonant
frequencies in phase-space $T^{\ast}\mathbb{T}^{d}$. This construction is
related to that of the so-called two-microlocal semiclassical measures. We
prove that any limit $\mu$ of the Wigner distributions corresponding to
solutions to the Schr\"{o}dinger equation on the torus is completely
determined by the limits of both the Wigner distribution and the resonant
Wigner distribution of the initial data; moreover, $\mu$ follows a propagation
law described by a family of density-matrix Schr\"{o}dinger equations on the
periodic geodesics of $\mathbb{T}^{d}$. Finally, we present some connections
with the study of the dispersive behavior of the Schr\"{o}dinger flow (in
particular, with Strichartz estimates). Among these, we show that the limits
of sequences of position densities of solutions to the Schr\"{o}dinger
equation on $\mathbb{T}^{2}$ are absolutely continuous with respect to the
Lebesgue measure.

\end{abstract}

\section{Introduction}

In this article, we shall consider solutions to Schr\"{o}dinger's equation on
the standard flat torus $\mathbb{T}^{d}:=\mathbb{R}^{d}/\left(  2\pi
\mathbb{Z}^{d}\right)  $,%
\begin{equation}
i\partial_{t}u\left(  t,x\right)  +\frac{1}{2}\Delta_{x}u\left(  t,x\right)
=0,\qquad\left(  t,x\right)  \in\mathbb{R}\times\mathbb{T}^{d}. \label{Schrod}%
\end{equation}
We are interested in understanding the propagation of high-frequency effects
associated to solutions to (\ref{Schrod}). More precisely, given a sequence
$\left(  u_{h}\right)  $ of initial data which oscillates at frequencies of
the order of $1/h$ (see condition (\ref{h-osc}) below) we would like to
describe in a quantitative manner the propagation of these oscillation effects
under the action of the Schr\"{o}dinger group $e^{it\Delta_{x}/2}$. Some
understanding in this direction can be obtained by analysing the structure of
weak $\ast$ limits of sequences of measures of the form:%
\begin{equation}
|e^{it\Delta_{x}/2}u_{h}\left(  x\right)  |^{2}dx, \label{position density}%
\end{equation}
where $\left(  u_{h}\right)  $ is a bounded sequence in $L^{2}\left(
\mathbb{T}^{d}\right)  $ and $dx$ denotes the Lebesgue measure on
$\mathbb{T}^{d}$. These limiting measures give information about the regions
on which the energy of $\left(  e^{it\Delta_{x}/2}u_{h}\right)  $
concentrates; a natural question in this context is to understand their
dependence on $t$; and in particular, their dependence on the initial data
$\left(  u_{h}\right)  $. However, it is usually difficult to deal directly
with (\ref{position density}). This is due to the presence of the modulus in
(\ref{position density}) which prevents us from keeping track of the
characteristic directions of oscillation of the functions $u_{h}$. It is
preferable instead to consider their Wigner distributions, which are
phase-space densities that take into account simultaneously concentration on
physical and Fourier space, and which project onto (\ref{position density}).
The main issue addressed in this article is to study the propagation laws
obeyed by Wigner distributions of solutions to the Schr\"{o}dinger equation
(\ref{Schrod}). As a consequence of our results, we shall prove that for $d=2$
a limit of a sequence of densities (\ref{position density}) (corresponding to
sequence $(u_{h})$ satisfying a standard oscillation condition) is absolutely
continuous with respect to the Lebesgue measure. At the end of this
introduction we discuss the connections of this result with Strichartz
estimates.\smallskip

Let $\psi_{k}\left(  x\right)  :=\left(  2\pi\right)  ^{-d/2}e^{ik\cdot x}$,
$k\in\mathbb{Z}^{d}$, denote the vectors of the standard orthonormal basis of
$L^{2}\left(  \mathbb{T}^{d}\right)  $. The \emph{Wigner distribution }of a
function $u=\sum_{k\in\mathbb{Z}^{d}}\widehat{u}\left(  k\right)  \psi_{k}$ in
$L^{2}\left(  \mathbb{T}^{d}\right)  $ is defined for $h>0$ as:%
\begin{equation}
w_{u}^{h}\left(  x,\xi\right)  :=\sum_{k,j\in\mathbb{Z}^{d}}\widehat{u}\left(
k\right)  \overline{\widehat{u}\left(  j\right)  }\psi_{k}\left(  x\right)
\overline{\psi_{j}\left(  x\right)  }\delta_{\frac{h}{2}\left(  k+j\right)
}\left(  \xi\right)  , \label{WD}%
\end{equation}
where $\delta_{p}$ stands for the Dirac delta centered at the point $p$. The
distribution $w_{u}^{h}$ is in fact a measure on $T^{\ast}\mathbb{T}^{d}%
\cong\mathbb{T}^{d}\times\mathbb{R}^{d}$; one easily checks that:%
\[
\int_{\mathbb{R}^{d}}w_{u}^{h}\left(  x,d\xi\right)  =\left\vert u\left(
x\right)  \right\vert ^{2},\qquad\int_{\mathbb{T}^{d}}w_{u}^{h}\left(
dx,\xi\right)  =\sum_{k\in\mathbb{Z}^{d}}\left\vert \widehat{u}\left(
k\right)  \right\vert ^{2}\delta_{hk}\left(  \xi\right)  .
\]
Therefore, $w_{u}^{h}$ may be viewed as a microlocal lift of the density
$\left\vert u\left(  x\right)  \right\vert ^{2}$ to phase space $T^{\ast
}\mathbb{T}^{d}$: it allows to describe simultaneously the distribution of
energy of $u$ on physical and Fourier space. The distributions $w_{u}^{h}$ are
not positive, although their limits as $h\rightarrow0^{+}$ are indeed finite
positive Radon measures.\footnote{From now on, we denote the set of such
measures by $\mathcal{M}_{+}\left(  T^{\ast}\mathbb{T}^{d}\right)  $.} Let us
note that the definition of the Wigner distribution in a general compact
Riemannian manifold usually depends on various choices (coordinate charts,
partitions of unity) that however have no effect on their asymptotic behavior
as $h\rightarrow0^{+}$ (see for instance, \cite{GeLei, MaciaAvSchrod}, and the
references therein). Our formula (\ref{WD}) corresponds to identifying
elements in $L^{2}\left(  \mathbb{T}^{d}\right)  $ to those functions in
$L_{\text{loc}}^{2}\left(  \mathbb{R}^{d}\right)  $ that are $2\pi
\mathbb{Z}^{d}$-periodic and then consider their canonical Euclidean Wigner
distributions. One easily checks that this definition is consistent with the
generally accepted one. The reader can refer to the book \cite{Fo} for a
comprehensive discussion on Wigner distributions.\smallskip

We shall assume in the following that $\left(  u_{h}\right)  $ is bounded in
$L^{2}\left(  \mathbb{T}^{d}\right)  $ and write
\[
w_{u_{h}}^{h}\left(  t,\cdot\right)  :=w_{e^{it\Delta_{x}/2}u_{h}}^{h}.
\]
Then, after possibly extracting a subsequence, the associated Wigner
distributions at $t=0$ converge (see, for instance \cite{GerardMDM, Ge91c,
Li-Pau, GeLei}):%
\begin{equation}
w_{u_{h}}^{h}\left(  0,\cdot\right)  \rightharpoonup\mu_{0}\in\mathcal{M}%
_{+}\left(  T^{\ast}\mathbb{T}^{d}\right)  ,\qquad\text{as }h\rightarrow
0^{+}\text{ in }\mathcal{D}^{\prime}\left(  T^{\ast}\mathbb{T}^{d}\right)  .
\label{conv0}%
\end{equation}
The measure $\mu_{0}$ is usually called the \emph{semiclassical or Wigner
measure} of $\left(  u_{h}\right)  $. In addition, one can recover the
asymptotic behavior of $\left\vert u_{h}\right\vert ^{2}$ from $\mu_{0}$. More
precisely:%
\[
\left\vert u_{h}\right\vert ^{2}dx\rightharpoonup\int_{\mathbb{R}^{d}}\mu
_{0}\left(  \cdot,d\xi\right)  ,\qquad\text{vaguely as }h\rightarrow0^{+},
\]
provided the densities $\left\vert u_{h}\right\vert ^{2}$ converge and the
sequence $\left(  u_{h}\right)  $ verifies the $h$\emph{-oscillation
property}:\footnote{This condition expresses that no positive fraction of the
energy of the sequence $\left(  u_{h}\right)  $ concentrates at frequencies
asymptotically larger than $1/h$.}%
\begin{equation}
\limsup_{h\rightarrow0^{+}}\sum_{\left\vert k\right\vert >R/h}\left\vert
\widehat{u_{h}}\left(  k\right)  \right\vert ^{2}\rightarrow0,\qquad\text{as
}R\rightarrow\infty. \label{h-osc}%
\end{equation}

Wigner distributions $w_{u_{h}}^{h}\left(  t,\cdot\right)  $ associated to
solutions to the Schr\"{o}dinger equation are completely determined by those
of their initial data $w_{u_{h}}^{h}\left(  0,\cdot\right)  $ as they solve
the classical Liouville equation:%
\begin{equation}
\partial_{t}w_{u_{h}}^{h}\left(  t,x,\xi\right)  +\frac{\xi}{h}\cdot\nabla
_{x}w_{u_{h}}^{h}\left(  t,x,\xi\right)  =0. \label{Liouville}%
\end{equation}
As a consequence of this, it is possible to show that the rescaled Wigner
distributions $w_{u_{h}}^{h}\left(  ht,\cdot\right)  $ converge (after
possibly extracting a subsequence) locally uniformly in $t$ to the measure
$\left(  \phi_{-t}\right)  _{\ast}\mu_{0}$, where $\left(  \phi_{t}\right)
_{\ast}$ denotes the push-forward operator on measures induced by the geodesic
flow $\phi_{t}\left(  x,\xi\right)  :=\left(  x+t\xi,\xi\right)  $ on
$T^{\ast}\mathbb{T}^{d}$. This result, sometimes known as the \emph{classical
limit}, holds in a much more general setting, see for instance \cite{Ge91c,
Li-Pau, GMMP}.

Equation (\ref{Liouville}) provides interesting consequences when $\left(
u_{n}\right)  $ is a sequence of (normalised) eigenfunctions $-\Delta_{x}$.
Suppose that $-\Delta_{x}u_{n}=\lambda_{n}u_{n}$ and that $\lim_{n\rightarrow
\infty}\lambda_{n}=\infty$. Since $w_{u_{n}}^{h}$ is quadratic in $u_{n}$, we
have $w_{u_{n}}^{h}\left(  t,\cdot\right)  =w_{u_{n}}^{h}\left(
0,\cdot\right)  $, for every $t\in\mathbb{R}$. Setting $h:=1/\sqrt{\lambda
_{n}}$, (so that $\left(  u_{n}\right)  $ is $h$-oscillating) equation
(\ref{Liouville}) implies that any semiclassical measure $\mu_{0}$ of this
sequence is invariant by the geodesic flow, \emph{i.e.} $\left(  \phi
_{s}\right)  _{\ast}\mu_{0}=\mu_{0}$. Moreover, one can show that $\mu_{0}$ is
a probability measure concentrated on the unit cosphere bundle $S^{\ast
}\mathbb{T}^{d}$; in this context, semiclassical measures are usually called
\emph{quantum limits}. The problem of classifying all possible quantum limits
in $\mathbb{T}^{d}$ is very hard. Their structure has been clarified by
Jakobson \cite{Jak} for $d=2$. For arbitrary $d\geq1$, Bourgain has proved
(see again \cite{Jak}) that the projection of a quantum limit onto
$\mathbb{T}^{d}$ (which is an accumulation point of the measures $(\left\vert
u_{n}\right\vert ^{2})$) is absolutely continuous with respect to Lebesgue
measure. In particular, the sequence $\left(  u_{n}\right)  $ cannot
concentrate on sets of dimension lower than $d$. The article \cite{Jak} also
provides partial results when $d\geq3$ (see also \cite{BourgainEigF}).
\smallskip

When $\left(  u_{n}\right)  $ is not formed by eigenfunctions, it is not
anymore clear that $w_{u_{n}}^{h}\left(  t,\cdot\right)  $ converges. However,
we have the following result, proved in \cite{MaciaAvSchrod} for a general
compact manifold: \smallskip

\noindent\emph{Existence of limits.}\textbf{ }Given a bounded $h$-oscillating
sequence $\left(  u_{h}\right)  $ in $L^{2}\left(  \mathbb{T}^{d}\right)  $,
there exist a subsequence such that, for every $a\in C_{c}^{\infty}\left(
T^{\ast}\mathbb{T}^{d}\right)  $ and every $\varphi\in L^{1}\left(
\mathbb{R}\right)  $,%
\begin{equation}
\lim_{h^{\prime}\rightarrow0^{+}}\int_{\mathbb{R}\times T^{\ast}\mathbb{T}%
^{d}}\varphi\left(  t\right)  a\left(  x,\xi\right)  w_{u_{h^{\prime}}%
}^{h^{\prime}}\left(  t,dx,d\xi\right)  dt=\int_{\mathbb{R}\times T^{\ast
}\mathbb{T}^{d}}\varphi\left(  t\right)  a\left(  x,\xi\right)  \mu\left(
t,dx,d\xi\right)  dt, \label{conv}%
\end{equation}
where the limit $\mu$ is in $L^{\infty}\left(  \mathbb{R};\mathcal{M}%
_{+}\left(  T^{\ast}\mathbb{T}^{d}\right)  \right)  $.\footnote{Note that, in
contrast with the semiclassical limit, it is not true that $(w_{u_{h}}%
^{h}\left(  t,\cdot\right)  )$ (or any subsequence) converges to $\mu\left(
t,\cdot\right)  $, even almost everywhere.} Moreover,%
\begin{equation}
\lim_{h^{\prime}\rightarrow0^{+}}\int_{\mathbb{R}}\int_{\mathbb{T}^{d}}%
\varphi\left(  t\right)  m\left(  x\right)  |e^{it\Delta_{x}/2}u_{h^{\prime}%
}\left(  x\right)  |^{2}dxdt=\int_{\mathbb{R}}\int_{T^{\ast}\mathbb{T}^{d}%
}\varphi\left(  t\right)  m\left(  x\right)  \mu\left(  t,dx,d\xi\right)  dt,
\label{convposden}%
\end{equation}
for every $\varphi\in L^{1}\left(  \mathbb{R}\right)  $ and $m\in C\left(
\mathbb{T}^{d}\right)  $.

In addition, as one can check by taking limits in equation (\ref{Liouville}),
the invariance property satisfied by semiclassical measures corresponding to
sequences of eigenfunctions also holds in this more general setting.\smallskip

\noindent\emph{Invariance. }For a.e. $t\in\mathbb{R}$, the measure $\mu\left(
t,\cdot\right)  $ is invariant by the geodesic flow on $T^{\ast}\mathbb{T}%
^{d}$:%
\begin{equation}
\left(  \phi_{s}\right)  _{\ast}\mu\left(  t,\cdot\right)  =\mu\left(
t,\cdot\right)  ,\qquad\text{for every }s\in\mathbb{R}\text{.}
\label{invariance}%
\end{equation}
See also \cite{MaciaAvSchrod} for a proof in a more general context. In that
reference, a characterization of the propagation law for these measures in the
class of compact manifolds with periodic geodesic flow (the so-called
\emph{Zoll manifolds}) was given. In fact, a formula relating $\mu$ and
$\mu_{0}$ exists: $\mu\left(  t,\cdot\right)  $ equals the average of $\mu
_{0}$ along the geodesic flow for a.e. $t\in\mathbb{R}$ (which is well-defined
due to the periodicity of the geodesic flow); note, in particular, that $\mu$
is constant in time. This fits our setting when $d=1$; but when $d\geq2$, the
dynamics of the geodesic flow in the torus are more complex than those in Zoll
manifolds. In both cases, the geodesic flow is completely integrable. However,
the torus possesses geodesics of arbitrary large minimal periods, as well as
non-periodic, dense, geodesics.\smallskip

It will turn out that this added complexity will have an effect on our
problem. However, there is still a class of sequences of initial data for
which the measures $\mu$ and $\mu_{0}$ are related by an averaging process.
More precisely, $\mu$ is obtained by averaging $\mu_{0}$ when the initial data
do not see the set of \emph{resonant frequencies}:
\[
\Omega:=\left\{  \xi\in\mathbb{R}^{d}\text{ }:\text{ }k\cdot\xi=0\text{ for
some }k\in\mathbb{Z}^{d}\setminus\left\{  0\right\}  \right\}  .
\]

\begin{proposition}
[Non-resonant case, \cite{MaciaAvSchrod} Proposition 10]Suppose $\mu$ and
$\mu_{0}$ are given respectively by (\ref{conv}) and (\ref{conv0}) for some
sequence $\left(  u_{h}\right)  $ bounded in $L^{2}\left(  \mathbb{T}%
^{d}\right)  $ and verifying (\ref{h-osc}). If $\mu_{0}\left(  \mathbb{T}%
^{d}\times\Omega\right)  =0$ then, for a.e. $t\in\mathbb{R}$,
\[
\mu\left(  t,x,\xi\right)  =\frac{1}{\left(  2\pi\right)  ^{d}}\int
_{\mathbb{T}^{d}}\mu_{0}\left(  dy,\xi\right)  .
\]

\end{proposition}

Note that in this case, any limit (\ref{convposden}) of the densities
$|e^{it\Delta_{x}/2}u_{h}\left(  x\right)  |^{2}$ is a constant function in
$t$ and $x$.\footnote{Note also that for $d=1$, the condition $\mu_{0}\left(
\mathbb{T}^{d}\times\Omega\right)  $ $=0$ reduces to $\mu_{0}\left(  \left\{
\xi=0\right\}  \right)  =0$. This has to be interpreted as the requirement
that no positive fraction of the energy of the sequence of initial data
concentrates at frequencies asymptotically smaller than $1/h$.}\smallskip

\noindent\emph{The role of resonances.}\textbf{ }Therefore, all the
difficulties in analysing the structure of $\mu$ rely on understanding its
behavior when $\mu_{0}$ actually sees the set of resonant frequencies $\Omega
$. Given $\xi_{0}\in\Omega$, in \cite{MaciaAvSchrod}, Proposition 11,
sequences of initial data $\left(  u_{h}\right)  $ and $\left(  v_{n}\right)
$ are constructed such that both have $\left\vert \rho\left(  x\right)
\right\vert ^{2}dx\delta_{\xi_{0}}\left(  \xi\right)  $ as a semiclassical
measure. However, when $\rho\in L^{2}\left(  \mathbb{T}^{d}\right)  $ is
invariant in the $\xi_{0}$-direction, the corresponding limits (\ref{conv}) of
the evolved Wigner distributions are, respectively,%
\[
\left\vert e^{it\Delta_{x}/2}\rho\left(  x\right)  \right\vert ^{2}%
dx\delta_{\xi_{0}}\left(  \xi\right)  ,\quad\text{and}\quad\frac{1}{\left(
2\pi\right)  ^{d}}\int_{\mathbb{T}^{d}}\left\vert \rho\left(  y\right)
\right\vert ^{2}dy\delta_{\xi_{0}}\left(  \xi\right)  .
\]
Two conclusions can be extracted from this fact: (i) The measures $\mu\left(
t,\cdot\right)  $ may have a non-trivial dependence on $t$, which is not
directly related to the dynamics of the geodesic flow and, most importantly,
(ii) the measure $\mu_{0}$ corresponding to the initial data\emph{ no longer
determines uniquely} the measures $\mu\left(  t,\cdot\right)  $ corresponding
to the evolution.\smallskip

The structure of $\mu$ is therefore much more complex when $\mu_{0}$ sees the
set of resonant frequencies; before stating our main result, let us introduce
some notation.

We define set $\mathbb{W}$ of \emph{resonant directions} in $\mathbb{R}^{d}$
as follows. Consider the subset $\Omega_{1}\subset\Omega$ consisting of simple
resonances (that is, $\Omega_{1}$ is formed by the $\xi\in\Omega$ such that
$\lambda\xi\in\mathbb{Z}^{d}$ for some real $\lambda\neq0$). Consider the
equivalence relation $\sim$ on $\Omega_{1}\setminus\left\{  0\right\}  $
defined by $x\sim y$ if and only if $x,y\in\Omega_{1}\setminus\left\{
0\right\}  $ lie on a line through the origin. Define $\mathbb{W}$ as the set
of equivalence classes of $\sim$. In other words, $\mathbb{W}$ is the subset
of the real projective space $\mathbb{RP}^{d-1}$ obtained by projecting
$\Omega_{1}\setminus\left\{  0\right\}  \subset\mathbb{R}^{d}$ using the
canonical covering projection.

For each $\omega\in\mathbb{W}$ define,
\[
\gamma_{\omega}:=\left\{  t\nu_{\omega}\;:\;t\in\mathbb{R}\right\}
/2\pi\mathbb{Z}^{d}\subset\mathbb{T}^{d},\quad\text{where }\nu_{\omega}%
\in\omega.
\]
This is the (closed) geodesic of $\mathbb{T}^{d}$ issued from the point $0$ in
the direction $\nu_{\omega}$. There is a bijection between $\mathbb{W}$ and
the set of closed geodesics in $\mathbb{T}^{d}$ that pass through $0$. In what
follows, $L^{2}\left(  \gamma_{\omega}\right)  $ will denote the space of
(equivalence classes of) square integrable functions on $\gamma_{\omega}$ with
respect to arc-length measure.

Given $\omega\in\mathbb{W}$, we shall denote by $I_{\omega}\subset
\mathbb{R}^{d}$ the hyperplane through the origin orthogonal to $\omega$. The
structure of the measures $\mu\left(  t,\cdot\right)  $ is given by the next
result (in Theorem \ref{ThmMainComplete} in Section \ref{SecMain}, we precise
the nature of the propagation law for $\mu\left(  t,\cdot\right)  $).

\begin{theorem}
\label{ThmMain}Let $\left(  u_{h}\right)  $ be a bounded, $h$-oscillating
sequence in $L^{2}\left(  \mathbb{T}^{d}\right)  $ with semiclassical measure
$\mu_{0}$. Suppose that (\ref{conv}) holds for some measure $\mu$. Then, for
a.e. $t\in\mathbb{R}$\emph{, }%
\begin{equation}
\mu\left(  t,x,\xi\right)  =\sum_{\omega\in\mathbb{W}}\rho_{\omega}^{t}\left(
x,\xi\right)  +\frac{1}{\left(  2\pi\right)  ^{d}}\int_{\mathbb{T}^{d}}\mu
_{0}\left(  dy,\xi\right)  , \label{MainFormula}%
\end{equation}
where $\rho_{\omega}^{t}$, for $t\in\mathbb{R}$\emph{ }and $\omega
\in\mathbb{W}$, is a signed measure concentrated on\emph{ }$\mathbb{T}%
^{d}\times I_{\omega}$\emph{ }whose projection on the first component is
absolutely continuous with respect to the Lebesgue measure.

\noindent Moreover, the measure $\rho_{\omega}^{t}$ follows a propagation law
related to the Schr\"{o}dinger flow on $L^{2}\left(  \gamma_{\omega}\right)
$, and can be computed solely in terms of $\rho_{\omega}^{0}$, which is in
turn completely determined by the initial data $\left(  u_{h}\right)  $.
\end{theorem}

Note that the sum $\sum_{\omega\in\mathbb{W}}\rho_{\omega}^{t}$ is
concentrated on $\mathbb{T}^{d}\times\Omega$, and, as we shall prove in
Theorem \ref{ThmMainComplete}, $\rho_{\omega}^{t}$ is related to the trace of
a density matrix in $L^{2}\left(  \gamma_{\omega}\right)  $ evolving according
to the Schr\"{o}dinger equation on $\gamma_{\omega}$. The measure
$\rho_{\omega}^{0}$ will be obtained as the limit of a new object, the
\emph{resonant Wigner distribution} of the initial data $u_{h}$, which
describes the concentration of energy of the sequence $\left(  u_{h}\right)  $
over $\mathbb{T}^{d}\times\Omega$ at scales of order one. We introduce its
definition, along with a description of the properties that are relevant to
our analysis in the next section. Let us just mention that the resonant Wigner
distribution may be viewed as a two-microlocal object, in the spirit of the
$2$-microlocal semiclassical measures introduced by Fermanian-Kammerer
\cite{FK2Micro, FK}, G\'{e}rard and Fermanian-Kammerer \cite{FK-G}, Miller
\cite{MillerTh}, and Nier \cite{NierQS}.

In particular, $\rho_{\omega}^{0}$ might vanish even if $\mu_{0}\left(
\mathbb{T}^{d}\times I_{\omega}\right)  >0$. The condition for $\rho_{\omega
}^{0}$ to be zero is the following (see Proposition \ref{PropRhoZero} in
Section \ref{SecEx}).\smallskip

\emph{Suppose that }$\rho_{\omega}^{t}$\emph{ are given by formula
(\ref{MainFormula}). If }$\left(  u_{h}\right)  $\emph{ satisfies:}%
\[
\lim_{h\rightarrow0^{+}}\sum_{\left\vert k\cdot\nu_{\omega}\right\vert
<N}\left\vert \widehat{u_{h}}\left(  k\right)  \right\vert ^{2}=0,\qquad
\text{for every }N>0\text{,}%
\]
\emph{where }$\nu_{\omega}\in\omega$\emph{ is a unit vector, then }%
$\rho_{\omega}^{t}=0$\emph{ for every }$t\in\mathbb{R}$\emph{.}\smallskip

Using this characterization, we are able to describe the propagation of
wave-packet type solutions, see Proposition \ref{PropWP} in Section
\ref{SecEx}. We also give there an example of sequence $\left(  u_{h}\right)
$ for which some of the $\rho_{\omega}^{t}$ are non-zero.

As a consequence of formula (\ref{MainFormula}) we prove in Section
\ref{SecMain} the following result for the position densities
(\ref{position density}).

\begin{corollary}
\label{CorPosDen}Let $d=2$ and $\left(  u_{h}\right)  $ be a bounded,
$h$-oscillating sequence in $L^{2}\left(  \mathbb{T}^{2}\right)  $ with a
semiclassical measure $\mu_{0}$. If $\mu_{0}\left(  \left\{  \xi=0\right\}
\right)  =0$ then, up to some subsequence, for every $\varphi\in L^{1}\left(
\mathbb{R}\right)  $ and $m\in C\left(  \mathbb{T}^{2}\right)  $,%
\[
\lim_{h^{\prime}\rightarrow0^{+}}\int_{\mathbb{R}}\int_{\mathbb{T}^{2}}%
\varphi\left(  t\right)  m\left(  x\right)  |e^{it\Delta_{x}/2}u_{h^{\prime}%
}\left(  x\right)  |^{2}dxdt=\int_{\mathbb{R}}\int_{\mathbb{T}^{2}}%
\varphi\left(  t\right)  m\left(  x\right)  \nu\left(  t,dx\right)  dt,
\]
and the measure $\nu\in L^{\infty}\left(  \mathbb{R};\mathcal{M}_{+}\left(
\mathbb{T}^{2}\right)  \right)  $ is absolutely continuous with respect to the
Lebesgue measure in $\mathbb{T}^{2}$.
\end{corollary}

This result is somehow related to the analysis of dispersion (Strichartz)
estimates for the Schr\"{o}dinger equation. For instance, when $d=1$ we have
the following inequality due to Zygmund \cite{Zyg74}:%
\begin{equation}
||e^{it\Delta/2}u||_{L^{4}\left(  \mathbb{T}_{t}\times\mathbb{T}_{x}\right)
}\leq C\left\Vert u\right\Vert _{L^{2}\left(  \mathbb{T}\right)  },
\label{strich}%
\end{equation}
for some constant $C>0$ (note that $e^{it\Delta/2}$ is $2\pi\mathbb{Z}%
$-periodic in $t$). This estimate implies that for $d=1$ any limit of averages
in time of the densities (\ref{position density}) is absolutely continuous
with respect to Lebesgue measure (and is in fact an $L^{2}\left(
\mathbb{T}\right)  $-function). However, as shown by Bourgain
\cite{BourgainFRLat}, estimate (\ref{strich}) fails when $d=2$; although a
version of (\ref{strich}) with a loss of derivatives does
hold.\footnote{Reference \cite{BourgainFRLat, BourgainBook} describes also
positive results. In \cite{BGT04, BGT05-1, BGT05} Strichartz estimates in
general compact manifolds are established, together with a detailed analysis
of the loss of derivatives phenomenon in specific geometries.} Therefore, the
result given in Corollary \ref{CorPosDen} supports in some sense the
possibility that an inequality such as (\ref{strich}) holds on $\mathbb{T}%
_{t}\times\mathbb{T}_{x}^{2}$ in some $L^{p}$-space with $2<p<4$.

The present analysis can be extended to more general tori and
Schr\"{o}dinger-type equations arising as the quantization of completely
integrable Hamiltonian systems. These issues will be addressed elsewhere.

\section{The resonant Wigner distribution}

\subsection{Preliminaries and definition}

Let $\omega\in\mathbb{W}$ be a resonant direction; as before, denote by
$I_{\omega}\subset\mathbb{R}^{d}$ the linear hyperplane orthogonal to $\omega
$. Then there exists a unique $p_{\omega}\in\omega\cap\mathbb{Z}^{d}$ such that:

\begin{enumerate}
\item[(i)] the (non-zero) components of $p_{\omega}$ are coprime;

\item[(ii)] the first non zero component of $p_{\omega}$ is positive.
\end{enumerate}

Clearly, $\omega\cap\mathbb{Z}^{d}=\left\{  np_{\omega}\;:\;n\in
\mathbb{Z}\right\}  $; therefore, the component in the direction $p_{\omega}$
of any $k\in\mathbb{Z}^{d}$ is of the form
\[
\frac{n}{\left\vert p_{\omega}\right\vert }\nu_{\omega},\quad\text{where
}n=k\cdot p_{\omega}\in\mathbb{Z}\text{ and }\nu_{\omega}:=\frac{p_{\omega}%
}{\left\vert p_{w}\right\vert }.
\]
Moreover, since Bezout's theorem ensures the existence of $c\in\mathbb{Z}^{d}$
satisfying $p_{\omega}\cdot c=1$, we have that for any given $n\in\mathbb{Z}$
there exists at least one $k\in\mathbb{Z}^{d}$ such that $k\cdot p_{\omega}%
=n$. In other words, the set of orthogonal projections onto $\omega$ of points
in $\mathbb{Z}^{d}$ consists of the vectors $n/\left\vert p_{\omega
}\right\vert \nu_{\omega}$ for $n\in\mathbb{Z}$. Note that, in particular, the
sets
\[
\omega_{n}^{\perp}:=\left\{  r\in I_{\omega}\;:\;\frac{n}{\left\vert
p_{\omega}\right\vert }\nu_{\omega}+r\in\mathbb{Z}^{d}\right\}  \subset
I_{\omega},
\]
are non-empty.

It is not difficult to see that $\omega_{n}^{\perp}\cap\omega_{m}^{\perp}%
\neq\emptyset$ if and only if $m\equiv n$ $($mod$\left\vert p_{\omega
}\right\vert ^{2})$, in which case $\omega_{n}^{\perp}=\omega_{m}^{\perp}$.
This implies that the set $\omega^{\perp}:=\bigcup_{n\in\mathbb{Z}}\omega
_{n}^{\perp}$ consisting of the orthogonal projections on $I_{\omega}$ of
vectors in $\mathbb{Z}^{d}$ is a subgroup of $\mathbb{R}^{d}$.

The results discussed so far imply the following.

\begin{proposition}
\label{RmkBij}For $p\in\mathbb{Z}$, denote by $\mathbb{Z}_{p}:=\mathbb{Z}%
/p\mathbb{Z}$ the group of congruence classes modulo $p$. The map:%
\[
\bigcup\limits_{\left[  c\right]  \in\mathbb{Z}_{\left\vert p_{\omega
}\right\vert ^{2}}}\left(  \left[  c\right]  \times\omega_{c}^{\perp}\right)
\rightarrow\mathbb{Z}^{d}:\left(  n,r\right)  \mapsto\frac{n}{\left\vert
p_{\omega}\right\vert }\nu_{\omega}+r
\]
is well-defined and bijective. Moreover, the map $h:\omega^{\perp}%
\rightarrow\mathbb{Z}_{\left\vert p_{\omega}\right\vert ^{2}}$ defined by
$h\left(  r\right)  :=\left[  n\right]  $ if $r\in\omega_{n}^{\perp}$ is a
well-defined group homomorphism whose kernel is $\omega_{0}^{\perp}%
\subset\mathbb{Z}^{d}$. Therefore, the quotient group $\omega^{\perp}%
/\omega_{0}^{\perp}$ is isomorphic to $\mathbb{Z}_{\left\vert p_{\omega
}\right\vert ^{2}}$ and consists of the cosets $\omega_{n}^{\perp}%
=r+\omega_{0}^{\perp}$ where $r$ is any element of $\omega_{n}^{\perp}$.
\end{proposition}

The geodesic $\gamma_{\omega}$, passing through $0$ and pointing in the
direction $\omega$, has length $2\pi\left\vert p_{\omega}\right\vert $.
Therefore, it can be identified to $\mathbb{T}_{\omega}:=\mathbb{R/}\left(
2\pi\left\vert p_{\omega}\right\vert \mathbb{Z}\right)  $ in such a way that
arc-length measure on $\gamma_{\omega}$ corresponds to a (suitably normalized)
Haar measure on $\mathbb{T}_{\omega}$. The functions:%
\[
\phi_{m}^{\omega}\left(  s\right)  :=\frac{e^{i\frac{m}{\left\vert p_{\omega
}\right\vert }s}}{\sqrt{2\pi\left\vert p_{\omega}\right\vert }},\quad
m\in\mathbb{Z},
\]
for an orthonormal basis of $L^{2}\left(  \gamma_{\omega}\right)  $. For
$n,m\in\mathbb{Z}$, we shall denote by $\phi_{m}^{\omega}\otimes\overline
{\phi_{n}^{\omega}}$ the operator on $L^{2}\left(  \gamma_{\omega}\right)  $
given by:%
\[
\phi_{m}^{\omega}\otimes\overline{\phi_{n}^{\omega}}\left(  \phi_{k}^{\omega
}\right)  =\left\{
\begin{array}
[c]{l}%
\phi_{m}^{\omega}\text{ if }k=n,\smallskip\\
0\text{ otherwise.}%
\end{array}
\right.
\]
We shall denote by $\mathcal{L}\left(  L^{2}\left(  \gamma_{\omega}\right)
\right)  $, $\mathcal{K}\left(  L^{2}\left(  \gamma_{\omega}\right)  \right)
$ and $\mathcal{L}^{1}\left(  L^{2}\left(  \gamma_{\omega}\right)  \right)  $,
the space of linear bounded, compact and trace-class operators on
$L^{2}\left(  \gamma_{\omega}\right)  $, respectively.$\smallskip$

We write:%
\[
\mathcal{J}:=\bigcup\limits_{\omega\in\mathbb{W}}\left\{  \omega\right\}
\times I_{\omega};
\]
consider on each $\left\{  \omega\right\}  \times I_{\omega}$ the topology
induced by $\mathbb{R}^{d}$ and endow $\mathcal{J}$ with the disjoint union
topology. To every $\left(  \omega,\xi\right)  \in\mathcal{J}$ we associate
the vector spaces $\mathcal{L}\left(  L^{2}\left(  \gamma_{\omega}\right)
\right)  $ and $\mathcal{K}\left(  L^{2}\left(  \gamma_{\omega}\right)
\right)  $; this defines vector bundles over $\mathcal{J}$:%
\[
\pi_{\mathcal{L}}:\bigcup\limits_{\left(  \omega,\xi\right)  \in\mathcal{J}%
}\mathcal{L}\left(  L^{2}\left(  \gamma_{\omega}\right)  \right)
\rightarrow\mathcal{J},\quad\pi_{\mathcal{K}}:\bigcup\limits_{\left(
\omega,\xi\right)  \in\mathcal{J}}\mathcal{K}\left(  L^{2}\left(
\gamma_{\omega}\right)  \right)  \rightarrow\mathcal{J}.
\]
Let us denote respectively by $\mathcal{X}\left(  \mathcal{J}\right)  $ and
$\mathcal{X}_{0}\left(  \mathcal{J}\right)  $ the spaces of continuous,
compactly supported, sections of the bundles $\pi_{\mathcal{L}}$ and
$\pi_{\mathcal{K}}.$ That is, $k\in\mathcal{X}\left(  \mathcal{J}\right)  $
whenever $k\left(  \omega,\xi\right)  \in\mathcal{L}\left(  L^{2}\left(
\gamma_{\omega}\right)  \right)  $ for every $\left(  \omega,\xi\right)
\in\mathcal{J}$ and $k\left(  \omega,\cdot\right)  $ is continuous, compactly
supported, and non-zero for at most a finite number of $\omega$. Similar
considerations hold for the elements of $\mathcal{X}_{0}\left(  \mathcal{J}%
\right)  $. The dual of $\mathcal{X}\left(  \mathcal{J}\right)  $ (resp.
$\mathcal{X}_{0}\left(  \mathcal{J}\right)  $) will be denoted by
$\mathcal{X}^{\prime}\left(  \mathcal{J}\right)  $ (resp. $\mathcal{X}%
_{0}^{\prime}\left(  \mathcal{J}\right)  $). Given $\mu\in\mathcal{X}^{\prime
}\left(  \mathcal{J}\right)  $ (resp. $\mathcal{X}_{0}^{\prime}\left(
\mathcal{J}\right)  $), $\mu\left(  \omega,\cdot\right)  $ can be identified
to a measure on $I_{\omega}$ taking values on $\mathcal{L}\left(  L^{2}\left(
\gamma_{\omega}\right)  \right)  $ (resp. $\mathcal{L}^{1}\left(  L^{2}\left(
\gamma_{\omega}\right)  \right)  $). Finally, $\mathcal{X}_{0,+}^{\prime
}\left(  \mathcal{J}\right)  $ will stand for the cone of positive elements of
$\mathcal{X}_{0}^{\prime}\left(  \mathcal{J}\right)  $; hence, $\mu
\in\mathcal{X}_{0}^{\prime}\left(  \mathcal{J}\right)  $ if $\int_{I_{\omega}%
}b\left(  \xi\right)  \mu\left(  \omega,d\xi\right)  \in\mathcal{L}^{1}\left(
L^{2}\left(  \gamma_{\omega}\right)  \right)  $ is positive and Hermitian
whenever $b\in C_{c}\left(  I_{\omega}\right)  $ is positive. The Appendix
provides background material, additional details and references on
operator-valued measures.$\smallskip$

The \emph{resonant Wigner distribution }of $u\in L^{2}\left(  \mathbb{T}%
^{d}\right)  $, is defined as:%
\begin{equation}
\mathcal{R}_{u}^{h}\left(  \omega,\xi\right)  :=\sum_{\left[  c\right]
\in\mathbb{Z}_{\left\vert p_{\omega}\right\vert ^{2}}}\sum_{\substack{m,n\in
\left[  c\right]  \\r\in\omega_{c}^{\perp}}}\widehat{u}\left(  \frac
{m}{\left\vert p_{\omega}\right\vert }\nu_{\omega}+r\right)  \overline
{\widehat{u}\left(  \frac{n}{\left\vert p_{\omega}\right\vert }\nu_{\omega
}+r\right)  }\delta_{hr}\left(  \xi\right)  \phi_{m}^{\omega}\otimes
\overline{\phi_{n}^{\omega}}. \label{WD2m}%
\end{equation}
Clearly, $\mathcal{R}_{u}^{h}\in\mathcal{X}^{\prime}\left(  \mathcal{J}%
\right)  $.

\subsection{Boundedness and convergence}

Our next result shows, in particular, that $\mathcal{R}_{u}^{h}\in
\mathcal{X}_{0,+}^{\prime}\left(  \mathcal{J}\right)  $.

\begin{proposition}
Let $u\in L^{2}\left(  \mathbb{T}^{d}\right)  $. Then for every $\omega
\in\mathbb{W}$ and $b\in C_{c}\left(  I_{\omega}\right)  $,%
\[
\int_{I_{\omega}}b\left(  \xi\right)  \mathcal{R}_{u}^{h}\left(  \omega
,d\xi\right)  \quad\text{is an Hermitian, trace-class operator of }%
L^{2}\left(  \gamma_{\omega}\right)  \text{,}%
\]
which is positive if $b$ is non-negative. In addition, we have the following
bound:%
\begin{equation}
\operatorname{tr}\left\vert \int_{I_{\omega}}b\left(  \xi\right)
\mathcal{R}_{u}^{h}\left(  \omega,d\xi\right)  \right\vert \leq\left\Vert
u\right\Vert _{L^{2}\left(  \mathbb{T}^{d}\right)  }^{2}\sup_{r\in I_{\omega}%
}\left\vert b\left(  r\right)  \right\vert . \label{TrBdd}%
\end{equation}
In particular, $\mathcal{R}_{u}^{h}\in\mathcal{X}_{0,+}^{\prime}\left(
\mathcal{J}\right)  $.
\end{proposition}

\begin{proof}
Define%
\[
\left\Vert u\right\Vert _{m,\omega}^{2}:=\sum_{r\in\omega_{m}^{\perp}%
}\left\vert \widehat{u}\left(  \frac{m}{\left\vert p_{\omega}\right\vert }%
\nu_{\omega}+r\right)  \right\vert ^{2};
\]
because of Proposition \ref{RmkBij}, one has $\left\Vert u\right\Vert
_{L^{2}\left(  \mathbb{T}^{d}\right)  }^{2}=\sum_{m\in\mathbb{Z}}\left\Vert
u\right\Vert _{m,\omega}^{2}$. Take $b\in C_{c}\left(  I_{\omega}\right)  $
and set
\[
K_{b}:=\int_{I_{\omega}}b\left(  \xi\right)  \mathcal{R}_{u}^{h}\left(
\omega,d\xi\right)  ;
\]
then $K_{b}=\sum_{m\equiv n\text{ }(\text{{\small mod}}\left\vert p_{\omega
}\right\vert ^{2})}k_{b}\left(  m,n\right)  \phi_{m}^{\omega}\otimes
\overline{\phi_{n}^{\omega}}$ with:%
\[
k_{b}\left(  m,n\right)  =\sum_{r\in\omega_{c}^{\perp}}b\left(  hr\right)
\widehat{u}\left(  \frac{m}{\left\vert p_{\omega}\right\vert }\nu_{\omega
}+r\right)  \overline{\widehat{u}\left(  \frac{n}{\left\vert p_{\omega
}\right\vert }\nu_{\omega}+r\right)  }.
\]
The operator $K_{b}$ is bounded, since:%
\begin{align*}
\sum_{m\equiv n\text{ }(\text{{\small mod}}\left\vert p_{\omega}\right\vert
^{2})}\left\vert k_{b}\left(  m,n\right)  \right\vert ^{2}  &  \leq\left\Vert
b\right\Vert _{L^{\infty}\left(  I_{\omega}\right)  }^{2}\sum_{m\equiv n\text{
}(\text{{\small mod}}\left\vert p_{\omega}\right\vert ^{2})}\left\Vert
u\right\Vert _{m,\omega}^{2}\left\Vert u\right\Vert _{n,\omega}^{2}\\
&  \leq\left\Vert b\right\Vert _{L^{\infty}\left(  I_{\omega}\right)  }%
^{2}\left(  \sum_{n\in\mathbb{Z}}\left\Vert u\right\Vert _{n,\omega}%
^{2}\right)  ^{2}=\left\Vert b\right\Vert _{L^{\infty}\left(  I_{\omega
}\right)  }^{2}\left\Vert u\right\Vert _{L^{2}\left(  \mathbb{T}^{d}\right)
}^{4}.
\end{align*}
Moreover, $K_{b}$ is Hermitian as soon as $b$ is real valued, since
$k_{b}\left(  m,n\right)  =\overline{k_{b}\left(  n,m\right)  }$; therefore,
$K_{b}$ is a Hilbert-Schmidt operator on $L^{2}\left(  \gamma_{\omega}\right)
$.

Now, given $v\in L^{2}\left(  \gamma_{\omega}\right)  $ write $v=\sum
_{m\in\mathbb{Z}}v_{m}\phi_{m}^{\omega}$. Then%
\[
\left(  K_{b}v|v\right)  _{L^{2}\left(  \gamma_{\omega}\right)  }%
=\sum_{\left[  c\right]  \in\mathbb{Z}_{\left\vert p_{\omega}\right\vert ^{2}%
}}\sum_{r\in\omega_{c}^{\perp}}b\left(  hr\right)  \left\vert \sum
_{n\in\left[  c\right]  }v_{n}\widehat{u}\left(  \frac{n}{\left\vert
p_{\omega}\right\vert }\nu_{\omega}+r\right)  \right\vert ^{2}.
\]
This quantity is positive whenever $b\geq0$ and $b\not \equiv 0$. Therefore,
for such $b$ the operator $K_{b}$ is Hilbert-Schmidt (and hence compact),
Hermitian and positive. Thus, it will be trace-class as soon as its trace is
finite. This is clearly the case, since%
\[
\operatorname{tr}K_{b}=\sum_{n\in\mathbb{Z}}\sum_{r\in\omega_{n}^{\perp}%
}b\left(  hr\right)  \left\vert \widehat{u}\left(  \frac{n}{\left\vert
p_{\omega}\right\vert }\nu_{\omega}+r\right)  \right\vert ^{2}\leq\sup_{r\in
I_{\omega}}b\left(  r\right)  \left\Vert u\right\Vert _{L^{2}\left(
\mathbb{T}^{d}\right)  }^{2}.
\]
For a general $b$ non necessarily positive, the result follows by expressing
$b=b_{+}-b_{-}$ and applying the above estimate to each term separately.
\end{proof}

If $\left(  u_{h}\right)  $ is a bounded family in $L^{2}\left(
\mathbb{T}^{d}\right)  $, estimate (\ref{TrBdd}) then shows that
$\mathcal{R}_{u_{h}}^{h}\left(  \omega,\cdot\right)  $ is an uniformly bounded
family in $\mathcal{X}_{0,+}^{\prime}\left(  \mathcal{J}\right)  $.

\begin{proposition}
\label{PropConv2m}Let $\left(  u_{h}\right)  $ be a bounded sequence in
$L^{2}\left(  \mathbb{T}^{d}\right)  $. Then, there exists a subsequence
$\left(  u_{h^{\prime}}\right)  $ and a finite measure $\mu_{\mathcal{R}}%
\in\mathcal{X}_{0,+}^{\prime}\left(  \mathcal{J}\right)  $, such that, for
every $\omega\in\mathbb{W}$ and $b\in\mathcal{X}_{0}\left(  \mathcal{J}%
\right)  $:%
\begin{equation}
\lim_{h^{\prime}\rightarrow0^{+}}\operatorname{tr}\int_{I_{\omega}}b\left(
\omega,\xi\right)  \mathcal{R}_{u_{h^{\prime}}}^{h^{\prime}}\left(
\omega,d\xi\right)  =\operatorname{tr}\int_{\mathbb{R}^{d}}b\left(  \omega
,\xi\right)  \mu_{\mathcal{R}}\left(  \omega,d\xi\right)  . \label{conv2mirco}%
\end{equation}
Moreover, the total mass of $\mu_{\mathcal{R}}\left(  \omega,\cdot\right)  $
satisfies:%
\[
\operatorname{tr}\int_{I_{\omega}}\mu_{\mathcal{R}}\left(  \omega,d\xi\right)
\leq\liminf_{h^{\prime}\rightarrow0^{+}}\left\Vert u_{h^{\prime}}\right\Vert
_{L^{2}\left(  \mathbb{T}^{d}\right)  }\text{;}%
\]
and, due to the structure of $\mathcal{R}_{u_{h}}^{h}$,%
\[
(%
{\textstyle\int_{I_{\omega}}}
b\left(  \xi\right)  \mu_{\mathcal{R}}\left(  \omega,d\xi\right)  \phi
_{n}^{\omega}|\phi_{m}^{\omega})_{L^{2}\left(  \gamma_{\omega}\right)
}=0\text{,\quad if }m\not \equiv n\text{ }(\text{\emph{mod}}\left\vert
p_{\omega}\right\vert ^{2}).
\]

\end{proposition}

\begin{proof}
Estimate (\ref{TrBdd}) implies that each $\mathcal{R}_{u_{h}}^{h}\left(
\omega,\cdot\right)  $ is uniformly bounded in the space\footnote{We refer the
reader to the Appendix for precise definitions of spaces of operator-valued
measures.} $\mathcal{M}_{+}(I_{\omega};\allowbreak\mathcal{L}^{1}(L^{2}\left(
\gamma_{\omega}\right)  ))$ of positive measures on $I_{\omega}$ with values
in $\mathcal{L}^{1}\left(  L^{2}\left(  \gamma_{\omega}\right)  \right)  $ by
a constant $C>0$ independent of $\omega\in\mathbb{W}$. Since $\mathcal{M}%
_{+}\left(  I_{\omega};\mathcal{L}^{1}\left(  L^{2}\left(  \gamma_{\omega
}\right)  \right)  \right)  $ may be identified to the cone of positive
elements of the dual of $C_{c}\left(  I_{\omega};\mathcal{K}\left(
L^{2}\left(  \gamma_{\omega}\right)  \right)  \right)  $, statement
(\ref{conv2mirco}) follows from the Banach-Alaoglu theorem and a standard
diagonal argument. Finally, the bound on the total mass of $\mu_{\mathcal{R}%
}\left(  \omega,\cdot\right)  $ is a consequence of estimate (\ref{TrBdd}),
\end{proof}

In what follows, we shall refer to a measure $\mu_{\mathcal{R}}\in
\mathcal{X}_{0,+}^{\prime}\left(  \mathcal{J}\right)  $ obtained as a limit
(\ref{conv2mirco}) as a \emph{resonant Wigner measure} of the sequence
$\left(  u_{h}\right)  $.

As we mentioned in the introduction, resonant Wigner measures are closely
related to the two-microlocal semiclassical measures introduced in
\cite{FK2Micro, FK, FK-G, MillerTh, NierQS}. However, our definition of the
resonant Wigner distribution gives rise to a global object (see the discussion
in \cite{FK2Micro}, p. 518); moreover, resonant Wigner measures describe the
energy concentration (at scales of order one) of the sequence $\left(
u_{h}\right)  $ on the non-smooth set $\mathbb{T}^{d}\times\Omega$ in phase space.

\subsection{Additional properties}

Our next result is a manifestation of the two-microlocal character of resonant
Wigner measures. It characterizes the sequences $\left(  u_{h}\right)  $ for
which $\mu_{\mathcal{R}}$ is identically zero.

\begin{proposition}
\label{Prop2mircoCon}Let $\left(  u_{h}\right)  $ be $h$-oscillatory and
suppose that (\ref{conv2mirco}) holds for the sequence $\left(  u_{h}\right)
$. Given any $\omega\in\mathbb{W}$, one has $\mu_{\mathcal{R}}\left(
\omega,\cdot\right)  =0$ if and only if:%
\begin{equation}
\lim_{h\rightarrow0^{+}}\sum_{\left\vert k\cdot p_{\omega}\right\vert
<N}\left\vert \widehat{u_{h}}\left(  k\right)  \right\vert ^{2}=0,\qquad
\text{for every }N>0\text{.} \label{cosc}%
\end{equation}

\end{proposition}

\begin{proof}
Let $N\in\mathbb{N}$ and denote by $\pi_{N}$ the projection in $L^{2}\left(
\gamma_{\omega}\right)  $ onto the subspace spanned by $\left(  \phi
_{j}^{\omega}\right)  _{0\leq\left\vert j\right\vert \leq N}$. Then $\pi_{N}$
is compact and%
\[
\operatorname{tr}\left(  \pi_{N}\mathcal{R}_{u_{h}}^{h}\left(  \omega
,\xi\right)  \right)  =\sum_{\left\vert n\right\vert \leq N}\sum_{r\in
\omega_{n}^{\perp}}\left\vert \widehat{u_{h}}\left(  \frac{n}{\left\vert
p_{\omega}\right\vert }\nu_{\omega}+r\right)  \right\vert ^{2}\delta
_{hr}\left(  \xi\right)  .
\]
Therefore, for every $\varphi\in C_{c}\left(  I_{\omega}\right)  $,%
\begin{equation}
\lim_{h\rightarrow0^{+}}\sum_{\left\vert n\right\vert \leq N}\sum_{r\in
\omega_{n}^{\perp}}\varphi\left(  hr\right)  \left\vert \widehat{u_{h}}\left(
\frac{n}{\left\vert p_{\omega}\right\vert }\nu_{\omega}+r\right)  \right\vert
^{2}=\operatorname{tr}\left(  \pi_{N}\int_{I_{\omega}}\varphi\left(
\xi\right)  \mu_{\mathcal{R}}\left(  \omega,d\xi\right)  \right)  .
\label{conormalosc}%
\end{equation}
Since $\left(  u_{h}\right)  $ is $h$-oscillatory, we can suppose without loss
of generality that there exists $R>0$ such that $\widehat{u_{h}}\left(
k\right)  =0$ if $\left\vert hk\right\vert >R$. Now suppose $\mu_{\mathcal{R}%
}\left(  \omega,\cdot\right)  =0$; by taking $\varphi\left(  \xi\right)  =1$
for $\left\vert \xi\right\vert \leq R$ in (\ref{conormalosc}) we conclude
(\ref{cosc}). Now suppose that (\ref{cosc}) holds. Then $\operatorname{tr}%
\left(  \pi_{N}\int_{I_{\omega}}\varphi\left(  \xi\right)  \mathcal{R}_{u_{h}%
}^{h}\left(  \omega,d\xi\right)  \right)  =0$ for every $\varphi\in
C_{c}\left(  I_{\omega}\right)  $ and every $N>0$. By letting $N$ tend to
infinity we conclude that $\operatorname{tr}\int_{I_{\omega}}\varphi\left(
\xi\right)  \mu_{\mathcal{R}}\left(  \omega,d\xi\right)  =0$, which implies,
since $\varphi$ is arbitrary, $\mu_{\mathcal{R}}\left(  \omega,\cdot\right)
=0$.
\end{proof}

Since the energy of sequence $\left(  u_{h}\right)  $ may concentrate on
$I_{\omega}$ at scales larger than one, typically the restriction to
$I_{\omega}$ of the semiclassical measure of $\left(  u_{h}\right)  $ is
larger than $\operatorname{tr}\mu_{\mathcal{R}}\left(  \omega,\xi\right)  $.
This is the content of our next result.

\begin{proposition}
Let $\left(  u_{h}\right)  $ have a semiclassical measure $\mu_{0}$ and
satisfy (\ref{conv2mirco}). Then, for every non-negative $\varphi\in
C_{c}\left(  \mathbb{R}^{d}\right)  $ and every $\omega\in\mathbb{W}$,%
\begin{equation}
\operatorname{tr}\int_{I_{\omega}}\varphi\left(  \xi\right)  \mu_{\mathcal{R}%
}\left(  \omega,d\xi\right)  \leq\int_{T^{\ast}\mathbb{T}^{d}}\varphi\left(
\xi\right)  \mu_{0}\left(  dx,d\xi\right)  . \label{muRltm0}%
\end{equation}

\end{proposition}

\begin{proof}
Consider the projector $\pi_{N}$ defined in the proof of Proposition
\ref{Prop2mircoCon}. For every $\varphi\in C_{c}^{1}\left(  \mathbb{R}%
^{d}\right)  $ and every $N>0$ one has:%
\begin{align*}
&  \left\vert \operatorname{tr}\left(  \pi_{N}\int_{I_{\omega}}\varphi\left(
\xi\right)  \mathcal{R}_{u_{h}}^{h}\left(  \omega,d\xi\right)  \right)
-\sum_{\left\vert p_{\omega}\cdot k\right\vert \leq N}\varphi\left(
hk\right)  \left\vert \widehat{u_{h}}\left(  k\right)  \right\vert
^{2}\right\vert \\
&  \leq hN\left\Vert \nabla_{\xi}\varphi\right\Vert _{L^{\infty}\left(
I_{\omega}\right)  }\left\Vert u_{h}\right\Vert _{L^{2}\left(  \mathbb{T}%
^{d}\right)  }^{2}.
\end{align*}
Therefore, if in addition $\varphi$ is non-negative,%
\[
\operatorname{tr}\left(  \pi_{N}\int_{I_{\omega}}\varphi\left(  \xi\right)
\mathcal{R}_{u_{h}}^{h}\left(  \omega,d\xi\right)  \right)  \leq\sum
_{k\in\mathbb{Z}^{d}}\varphi\left(  hk\right)  \left\vert \widehat{u_{h}%
}\left(  k\right)  \right\vert ^{2}+\mathcal{O}\left(  h\right)  ;
\]
taking limits as $h\rightarrow0^{+}$ we get, for every $N>0$:%
\[
\operatorname{tr}\left(  \pi_{N}\int_{I_{\omega}}\varphi\left(  \xi\right)
\mu_{\mathcal{R}}\left(  \omega,d\xi\right)  \right)  \leq\int_{T^{\ast
}\mathbb{T}^{d}}\varphi\left(  \xi\right)  \mu_{0}\left(  dx,d\xi\right)  .
\]
Letting $N\rightarrow\infty$ and using the density of $C_{c}^{1}\left(
\mathbb{R}^{d}\right)  $ in $C_{c}\left(  \mathbb{R}^{d}\right)  $ we conclude
the proof of the proposition.
\end{proof}

\section{\label{SecMain}Resonant Wigner measures and the Schr\"{o}dinger flow}

Before stating our main result, we need some more notation. Given $\omega
\in\mathbb{W}$, denote by $L_{\omega}$ the length of $\gamma_{\omega}$, equal
to $2\pi\left\vert p_{\omega}\right\vert $. There is a well-defined extension
operator $\mathcal{E}_{\omega}$ from the space of $L_{\omega}/\left\vert
p_{\omega}\right\vert ^{2}\mathbb{Z}$-periodic functions in $L^{1}\left(
\gamma_{\omega}\right)  $ to the space $L^{1}\left(  \mathbb{T}^{d}\right)  $.
If $f\in L^{1}\left(  \gamma_{\omega}\right)  $ is $L_{\omega}/\left\vert
p_{\omega}\right\vert ^{2}\mathbb{Z}$-periodic then put%
\[
\mathcal{E}_{\omega}f\left(  x\right)  :=\frac{\left\vert p_{\omega
}\right\vert }{\left(  2\pi\right)  ^{d-1}}f\left(  x\cdot\nu_{\omega}\right)
,\quad x\in\mathbb{T}^{d}.
\]
With that normalization, it is not difficult to check that $\left\Vert
\mathcal{E}_{\omega}f\right\Vert _{L^{1}\left(  \mathbb{T}^{d}\right)
}=\left\Vert f\right\Vert _{L^{1}\left(  \gamma_{\omega}\right)  }$. Moreover,
$\mathcal{E}_{\omega}f$ is invariant by translations along vectors orthogonal
to $\omega$.

The following is a reformulation of Proposition \ref{PropAbsC} from the
Appendix in our setting. Let $\mu_{\mathcal{R}}\in\mathcal{X}_{0,+}\left(
\mathcal{J}\right)  $; for $f_{\omega}\in L^{\infty}\left(  \gamma_{\omega
}\right)  $ denote by $m_{f_{\omega}}$ the operator in $L^{2}\left(
\gamma_{\omega}\right)  $ defined by multiplication by $f_{\omega}$. The
measures $\tilde{\rho}_{\omega}\in\mathcal{M}\left(  \gamma_{\omega}\times
I_{\omega}\right)  $, $\omega\in\mathbb{W}$, defined by:%
\[
\int_{\gamma_{\omega}\times I_{\omega}}f_{\omega}\left(  s\right)
\varphi\left(  \xi\right)  \tilde{\rho}_{\omega}\left(  ds,d\xi\right)
:=\operatorname{tr}\left[  m_{f_{\omega}}\int_{I_{\omega}}\varphi\left(
\xi\right)  \mu_{\mathcal{R}}\left(  \omega,d\xi\right)  \right]  ,
\]
for $f_{\omega}\in C\left(  \gamma_{\omega}\right)  $ and $\varphi\in
C_{c}\left(  I_{\omega}\right)  $ are positive and $\tilde{\rho}_{\omega
}\left(  \cdot,\xi\right)  $ is absolutely continuous with respect to
arc-length measure $ds$ in $\gamma_{\omega}$. Clearly, $\int_{\gamma_{\omega}%
}\tilde{\rho}_{\omega}\left(  ds,\cdot\right)  =\operatorname{tr}%
\mu_{\mathcal{R}}\left(  \omega,\cdot\right)  $.

We shall say that $\tilde{\rho}_{\omega}$ is the \emph{trace density }of
$\mu_{\mathcal{R}}\left(  \omega,\cdot\right)  $. Let $\partial_{\omega}^{2}$
denote the Laplacian in $\mathbb{\gamma}_{\omega}$ (with respect to the
arc-length metric, that is, when $\gamma_{\omega}$ is identified to
$\mathbb{R}/\left(  2\pi\left\vert p_{\omega}\right\vert \mathbb{Z}\right)  $).

The next result complements Theorem \ref{ThmMain}.

\begin{theorem}
\label{ThmMainComplete}Let $\left(  u_{h}\right)  $ satisfy the hypotheses of
Theorem \ref{ThmMain}. Then the measures $\rho_{\omega}^{t}\in\mathcal{M}%
\left(  \mathbb{T}^{d}\times I_{\omega}\right)  $, $\omega\in\mathbb{W}$,
$t\in\mathbb{R}$, appearing in formula (\ref{MainFormula}) are uniquely
determined by the initial data $\left(  u_{h}\right)  $ as follows. Let
$\mu_{\mathcal{R}}^{0}\in\mathcal{X}_{0,+}\left(  \mathcal{J}\right)  $ be a
resonant Wigner measure corresponding to $\left(  u_{h}\right)  $. Let
$\mu_{\mathcal{R}}^{t}\in\mathcal{X}_{0,+}\left(  \mathcal{J}\right)  $,
$t\in\mathbb{R}$, solve the density matrix Schr\"{o}dinger equation:%
\begin{equation}
\left\{
\begin{array}
[c]{l}%
i\partial_{t}\mu_{\mathcal{R}}^{t}\left(  \omega,\xi\right)  =\left[
-\dfrac{1}{2}\partial_{\omega}^{2},\mu_{\mathcal{R}}^{t}\left(  \omega
,\xi\right)  \right]  ,\smallskip\\
\mu_{\mathcal{R}}^{t}|_{t=0}\left(  \omega,\xi\right)  =\mu_{\mathcal{R}}%
^{0}\left(  \omega,\xi\right)  .
\end{array}
\right.  \label{DMSchrod}%
\end{equation}
Let $\tilde{\rho}_{\omega}^{t}$ be the trace density of $\mu_{\mathcal{R}}%
^{t}\left(  \omega,\cdot\right)  $. Then%
\[
\rho_{\omega}^{t}\left(  \cdot,\xi\right)  =\mathcal{E}_{\omega}\left(
\tilde{\rho}_{\omega}^{t}\left(  \cdot,\xi\right)  -\operatorname{tr}%
\mu_{\mathcal{R}}^{t}\left(  \omega,\xi\right)  \right)  .
\]
In particular, each $\rho_{\omega}^{t}\left(  \cdot,\xi\right)  $ is
absolutely continuous with respect to Lebesgue measure in $\mathbb{T}^{d}$,
has zero total mass and is invariant under the geodesic flow.
\end{theorem}

Before giving the proof of Theorems \ref{ThmMain} and \ref{ThmMainComplete}.
we shall first need some preparatory lemmas.

Let $a\in C_{c}^{\infty}\left(  T^{\ast}\mathbb{T}^{d}\right)  $ and write it
as $a\left(  x,\xi\right)  =\sum_{k\in\mathbb{Z}^{d}}a_{k}\left(  \xi\right)
\psi_{k}\left(  x\right)  $. Given $\varphi\in\mathcal{S}\left(
\mathbb{R}\right)  $ define for every $\omega\in\mathbb{W}$ the
operator-valued function:%
\begin{equation}
k_{a,\varphi}\left(  \omega,\xi\right)  :=\frac{1}{\left(  2\pi\right)
^{d/2}}\sum_{\left[  c\right]  \in\mathbb{Z}_{\left\vert p_{\omega}\right\vert
^{2}}}\sum_{\substack{n,m\in\left[  c\right]  \\m\neq n}}\widehat{\varphi
}\left(  \frac{n^{2}-m^{2}}{2\left\vert p_{\omega}\right\vert ^{2}}\right)
a_{\frac{m-n}{\left\vert p_{\omega}\right\vert }\nu_{\omega}}\left(
\xi\right)  \phi_{m}^{\omega}\otimes\overline{\phi_{n}^{\omega}},
\label{kaphi}%
\end{equation}

\begin{lemma}
\label{LemmaEstHS}For every $\left(  \omega,\xi\right)  \in\mathcal{J}$, the
operator $k_{a,\varphi}\left(  \omega,\xi\right)  $ is a Hilbert-Schmidt
operator on $L^{2}\left(  \gamma_{\omega}\right)  $. Moreover,%
\begin{equation}
\sum_{\omega\in\mathbb{W}}\left\Vert k_{a,\varphi}\left(  \omega,\xi\right)
\right\Vert _{\mathcal{L}^{2}\left(  L^{2}\left(  \gamma_{\omega}\right)
\right)  }^{2}\leq\frac{1}{\left(  2\pi\right)  ^{d}}\sum_{n\in\mathbb{Z}^{d}%
}\left\vert \widehat{\varphi}\left(  \frac{n}{2}\right)  \right\vert
^{2}\left\Vert a\left(  \cdot,\xi\right)  -a_{0}\left(  \xi\right)  \psi
_{0}\right\Vert _{L^{2}\left(  \mathbb{T}^{d}\right)  }^{2}. \label{HSnorm}%
\end{equation}
In particular, $k_{a,\varphi}\in\mathcal{X}_{0}\left(  \mathcal{J}\right)  $.
\end{lemma}

\begin{proof}
A direct computation of the Hilbert-Schmidt norm gives:%
\begin{align*}
\left\Vert k_{a,\varphi}\left(  \omega,\xi\right)  \right\Vert _{\mathcal{L}%
^{2}\left(  L^{2}\left(  \gamma_{\omega}\right)  \right)  }^{2}  &  =\frac
{1}{\left(  2\pi\right)  ^{d}}\sum_{l\in\mathbb{Z\setminus}\left\{  0\right\}
}\left\vert a_{lp_{\omega}}\left(  \xi\right)  \right\vert ^{2}\sum
_{n\in\mathbb{Z}}\left\vert \widehat{\varphi}\left(  l\left(  \frac
{l\left\vert p_{\omega}\right\vert ^{2}}{2}+n\right)  \right)  \right\vert
^{2}\\
&  \leq\frac{1}{\left(  2\pi\right)  ^{d}}\sum_{l\in\mathbb{Z\setminus
}\left\{  0\right\}  }\left\vert a_{lp_{\omega}}\left(  \xi\right)
\right\vert ^{2}\sum_{n\in\mathbb{Z}}\left\vert \widehat{\varphi}\left(
\frac{n}{2}\right)  \right\vert ^{2},
\end{align*}
since, for every fixed $l\neq0$ the map that associates $n\in\mathbb{Z}$ to
$l(l\left\vert p_{\omega}\right\vert ^{2}+2n)\in\mathbb{Z}$ is injective.
Therefore, summing in $\omega$ gives the estimate (\ref{HSnorm}).\smallskip
\end{proof}

\begin{lemma}
\label{LemmaMain}Let $\left(  u_{h}\right)  $ be a bounded sequence in
$L^{2}\left(  \mathbb{T}^{d}\right)  $ such that (\ref{conv}) holds. Then,
given $a\in C_{c}^{\infty}\left(  T^{\ast}\mathbb{T}^{d}\right)  $ such that
$\int_{\mathbb{T}^{d}}a\left(  x,\cdot\right)  dx=0$ and $\varphi
\in\mathcal{S}\left(  \mathbb{R}\right)  $, we have:%
\[
\left\vert \int_{\mathbb{R}}\varphi\left(  t\right)  \left\langle w_{u_{h}%
}^{h}\left(  t,\cdot\right)  ,a\right\rangle dt-\sum_{\omega\in\mathbb{W}%
}\operatorname{tr}\left(  \int_{I_{\omega}}k_{a,\varphi}\left(  \omega
,\xi\right)  \mathcal{R}_{u_{h}}^{h}\left(  \omega,d\xi\right)  \right)
\right\vert \leq C_{a,\varphi}h,
\]
for some constant $C_{a,\varphi}>0$.
\end{lemma}

\begin{proof}
Let $\varphi\in\mathcal{S}\left(  \mathbb{R}\right)  $ and take $a\in
C_{c}^{\infty}\left(  T^{\ast}\mathbb{T}^{d}\right)  $ such that $a_{0}%
\equiv0$; from formula (\ref{WD}) we deduce:%
\[
\int_{\mathbb{R}}\varphi\left(  t\right)  \left\langle w_{u_{h}}^{h}\left(
t,\cdot\right)  ,a\right\rangle dt=\frac{1}{\left(  2\pi\right)  ^{d/2}}%
\sum_{k,j\in\mathbb{Z}^{d}}\widehat{\varphi}\left(  \frac{\left\vert
k\right\vert ^{2}-\left\vert j\right\vert ^{2}}{2}\right)  a_{j-k}\left(
h\frac{k+j}{2}\right)  \widehat{u_{h}}\left(  k\right)  \overline
{\widehat{u_{h}}\left(  j\right)  }.
\]
This expression can be written as:%
\begin{equation}
\frac{1}{\left(  2\pi\right)  ^{d/2}}\sum_{\omega\in\mathbb{W}}\sum
_{\substack{k,j\in\mathbb{Z}^{d}\\k-j\in\omega}}\widehat{\varphi}\left(
\frac{\left\vert k\right\vert ^{2}-\left\vert j\right\vert ^{2}}{2}\right)
a_{j-k}\left(  h\frac{k+j}{2}\right)  \widehat{u_{h}}\left(  k\right)
\overline{\widehat{u_{h}}\left(  j\right)  }, \label{WignerOmega}%
\end{equation}
since, recall $a_{0}\equiv0$, and, by definition, the lines $\omega
\in\mathbb{W}$ do not contain the origin. Now, for $\omega\in\mathbb{W}$
fixed, the sum in $k-j\in\omega$ in (\ref{WignerOmega}) may be rewritten in
terms of the parametrization introduced in Proposition \ref{RmkBij} to give:%
\begin{equation}
\sum_{_{\substack{\left[  c\right]  \in\mathbb{Z}_{\left\vert p_{\omega
}\right\vert ^{2}}\\\left(  m,n,r\right)  \in\mathcal{C}_{\omega,\left[
c\right]  }}}}\widehat{\varphi}\left(  \frac{m^{2}-n^{2}}{2\left\vert
p_{\omega}\right\vert ^{2}}\right)  a_{\frac{\left(  n-m\right)  }{\left\vert
p_{\omega}\right\vert }\nu_{\omega}}\left(  h\frac{m+n}{2\left\vert p_{\omega
}\right\vert }\nu_{\omega}+hr\right)  \widehat{u_{h}}\left(  \frac
{m}{\left\vert p_{\omega}\right\vert }\nu_{\omega}+r\right)  \overline
{\widehat{u_{h}}\left(  \frac{n}{\left\vert p_{\omega}\right\vert }\nu
_{\omega}+r\right)  }. \label{param}%
\end{equation}
where, for $\omega\in\mathbb{W}$ and $\left[  c\right]  \in\mathbb{Z}%
_{\left\vert p_{\omega}\right\vert ^{2}}$ we have set:%
\[
\mathcal{C}_{\omega,\left[  c\right]  }:=\left\{  \left(  m,n,r\right)
:m,n\in\left[  c\right]  ,\quad m\neq n,\quad r\in\omega_{c}^{\perp}\right\}
.
\]
Notice that the reason for (\ref{param}) to hold is that the condition
$k-j\in\omega$ results in the fact that $k$, $j$ can be written as $k=\frac
{m}{\left\vert p_{\omega}\right\vert }\nu_{\omega}+r$ and $j=\frac
{n}{\left\vert p_{\omega}\right\vert }\nu_{\omega}+r$ for a unique $\left(
m,n,r\right)  \in\bigcup_{\left[  c\right]  \in\mathbb{Z}_{\left\vert
p_{\omega}\right\vert ^{2}}}\mathcal{C}_{\omega,\left[  c\right]  }$.
Comparing (\ref{param}) with the expression (\ref{WD2m}) defining the resonant
Wigner distribution of $u_{h}$ we find that :%
\begin{align*}
&  \int_{\mathbb{R}}\varphi\left(  t\right)  \left\langle w_{u_{h}}^{h}\left(
t,\cdot\right)  ,a\right\rangle dt\\
&  =\sum_{\omega\in\mathbb{W}}\operatorname{tr}\left(  \int_{I_{\omega}%
}k_{a,\varphi}\left(  \omega,\xi\right)  \mathcal{R}_{u_{h}}^{h}\left(
\omega,d\xi\right)  \right)  +\sum_{\omega\in\mathbb{W}}\operatorname{tr}%
\left(  \int_{I_{\omega}}r_{a,\varphi}\left(  \omega,\xi\right)
\mathcal{R}_{u_{h}}^{h}\left(  \omega,d\xi\right)  \right)  ,
\end{align*}
where:%
\[
r_{a,\varphi}\left(  \omega,\xi\right)  :=\frac{1}{\left(  2\pi\right)
^{d/2}}\sum_{\substack{\left[  c\right]  \in\mathbb{Z}_{\left\vert p_{\omega
}\right\vert ^{2}}\\m,n\in\left[  c\right]  }}\widehat{\varphi}\left(
\frac{n^{2}-m^{2}}{2\left\vert p_{\omega}\right\vert ^{2}}\right)  l\left(
m,n,\omega,\xi\right)  \phi_{m}^{\omega}\otimes\overline{\phi_{n}^{\omega}},
\]
with
\[
l\left(  m,n,\omega,\xi\right)  :=a_{\frac{m-n}{\left\vert p_{\omega
}\right\vert }\nu_{\omega}}\left(  h\frac{m+n}{2\left\vert p_{\omega
}\right\vert }\nu_{\omega}+\xi\right)  -a_{\frac{m-n}{\left\vert p_{\omega
}\right\vert }\nu_{\omega}}\left(  \xi\right)  .
\]
Let us estimate the remainder term. First note that:%
\begin{equation}
\sup_{\xi\in I_{\omega}}\left\vert l\left(  m,n,\omega,\xi\right)  \right\vert
\leq\frac{h}{\left\vert p_{\omega}\right\vert }\left\vert \frac{m+n}%
{2}\right\vert \sup_{\xi\in I_{\omega}}\left\vert \nabla_{\xi}a_{\frac
{n-m}{\left\vert p_{\omega}\right\vert }\nu_{\omega}}\left(  \xi\right)
\right\vert . \label{est l}%
\end{equation}
Proceeding as in the proof of Lemma \ref{LemmaEstHS}, we use (\ref{est l}) to
estimate:%
\[
\left\Vert r_{a,\varphi}\left(  \omega,\xi\right)  \right\Vert _{\mathcal{L}%
^{2}\left(  L^{2}\left(  \gamma_{\omega}\right)  \right)  }^{2}\leq\frac
{h^{2}}{\left(  2\pi\right)  ^{d}}\sum_{n\in\mathbb{Z}}\left\vert \frac{n}%
{2}\widehat{\varphi}\left(  \frac{n}{2}\right)  \right\vert ^{2}\sum
_{l\in\mathbb{Z}\setminus\left\{  0\right\}  }\frac{\sup_{\xi\in I_{\omega}%
}\left\vert \nabla_{\xi}a_{lp_{\omega}}\left(  \xi\right)  \right\vert ^{2}%
}{l^{2}\left\vert p_{\omega}\right\vert ^{2}}.
\]
Since $a\in C_{c}^{\infty}\left(  T^{\ast}\mathbb{T}^{d}\right)  $ and
$a_{0}\equiv0$,
\[
\sum_{\omega\in\mathbb{W}}\left(  \sum_{l\in\mathbb{Z}\setminus\left\{
0\right\}  }l^{-2}\left\vert p_{\omega}\right\vert ^{-2}\sup_{\xi\in
I_{\omega}}\left\vert \nabla_{\xi}a_{lp_{\omega}}\left(  \xi\right)
\right\vert ^{2}\right)  ^{1/2}\text{ is finite.}%
\]
Therefore:%
\[
\left\vert \sum_{\omega\in\mathbb{W}}\operatorname{tr}\left(  \int_{I_{\omega
}}r_{a,\varphi}\left(  \omega,\xi\right)  \mathcal{R}_{u_{h}}^{h}\left(
\omega,d\xi\right)  \right)  \right\vert \leq C_{a,\varphi}\left\Vert
u_{h}\right\Vert _{L^{2}\left(  \mathbb{T}^{d}\right)  }^{2}h,
\]
and the result follows.\smallskip
\end{proof}

\begin{proof}
[Proof of Theorems \ref{ThmMain} and \ref{ThmMainComplete}]Suppose that
$\left(  \mathcal{R}_{u_{h^{\prime}}}^{h^{\prime}}\right)  $ converges along
some subsequence $\left(  u_{h^{\prime}}\right)  $ to the resonant Wigner
measure $\mu_{\mathcal{R}}^{0}$ (this is the case, by Proposition
\ref{PropConv2m}). Let $a\in C_{c}^{\infty}\left(  T^{\ast}\mathbb{T}%
^{d}\right)  $ and $\varphi\in\mathcal{S}\left(  \mathbb{R}\right)  $; in view
of estimate (\ref{HSnorm}) we have that:%
\[
\lim_{h^{\prime}\rightarrow0^{+}}\sum_{\omega\in\mathbb{W}}\operatorname{tr}%
\left(  \int_{I_{\omega}}k_{a,\varphi}\left(  \omega,\xi\right)
\mathcal{R}_{u_{h^{\prime}}}^{h^{\prime}}\left(  \omega,d\xi\right)  \right)
=\sum_{\omega\in\mathbb{W}}\operatorname{tr}\left(  \int_{I_{\omega}%
}k_{a,\varphi}\left(  \omega,\xi\right)  \mu_{\mathcal{R}}^{0}\left(
\omega,d\xi\right)  \right)  .
\]
Applying Lemma \ref{LemmaMain} we obtain:%
\begin{align}
&  \int_{\mathbb{R}}\int_{T^{\ast}\mathbb{T}^{d}}\varphi\left(  t\right)
a\left(  x,\xi\right)  \mu\left(  t,dx,d\xi\right) \label{e0}\\
&  =\sum_{\omega\in\mathbb{W}}\operatorname{tr}\int_{I_{\omega}}k_{a,\varphi
}\left(  \omega,\xi\right)  \mu_{\mathcal{R}}^{0}\left(  \omega,d\xi\right)
+\int_{T^{\ast}\mathbb{T}^{d}}\overline{a}\left(  \xi\right)  \mu_{0}\left(
dx,d\xi\right)  ,\nonumber
\end{align}
where $\overline{a}\left(  \xi\right)  :=\left(  2\pi\right)  ^{-d}%
\int_{\mathbb{T}^{d}}a\left(  x,\xi\right)  dx$. Therefore, it only remains to
identify the term involving the resonant Wigner measure $\mu_{\mathcal{R}}%
^{0}$.

Simple inspection gives:%
\begin{equation}
\operatorname{tr}\int_{I_{\omega}}k_{a,\varphi}\left(  \omega,\xi\right)
\mu_{\mathcal{R}}^{0}\left(  \omega,d\xi\right)  =\int_{\mathbb{R}}%
\varphi\left(  t\right)  \operatorname{tr}\int_{I_{\omega}}p_{a}\left(
\omega,\xi\right)  e^{it\partial_{\omega}^{2}/2}\mu_{\mathcal{R}}^{0}\left(
\omega,d\xi\right)  e^{-it\partial_{\omega}^{2}/2}, \label{e1}%
\end{equation}
where $\partial_{\omega}^{2}$ denotes the Laplacian on $L^{2}\left(
\gamma_{\omega}\right)  $ and%
\begin{equation}
p_{a}\left(  \omega,\xi\right)  :=\frac{1}{\left(  2\pi\right)  ^{d/2}}%
\sum_{\left[  c\right]  \in\mathbb{Z}_{\left\vert p_{\omega}\right\vert }}%
\sum_{\substack{n,m\in\left[  c\right]  \\m\neq n}}a_{\frac{m-n}{\left\vert
p_{\omega}\right\vert }\nu_{\omega}}\left(  \xi\right)  \phi_{m}^{\omega
}\otimes\overline{\phi_{n}^{\omega}}. \label{e2}%
\end{equation}
Note that
\[
\mu_{\mathcal{R}}^{t}\left(  \omega,\xi\right)  :=e^{it\partial_{\omega}%
^{2}/2}\mu_{\mathcal{R}}^{0}\left(  \omega,\xi\right)  e^{-it\partial_{\omega
}^{2}/2}%
\]
solves (\ref{DMSchrod}). For $m\equiv n$ (mod$\left\vert p_{\omega}\right\vert
^{2}$), define the measures $\mu_{\mathcal{R}}^{t}\left(  \omega,\xi\right)
\left(  m,n\right)  $:%
\begin{equation}
\int_{I_{\omega}}\varphi\left(  \xi\right)  \mu_{\mathcal{R}}^{t}\left(
\omega,d\xi\right)  \left(  m,n\right)  :=(%
{\textstyle\int_{I_{\omega}}}
\varphi\left(  \xi\right)  \mu_{\mathcal{R}}^{t}\left(  \omega,d\xi\right)
\phi_{n}^{\omega}|\phi_{m}^{\omega})_{L^{2}\left(  \gamma_{\omega}\right)
},\quad\varphi\in C_{c}\left(  I_{\omega}\right)  , \label{coordmeasures}%
\end{equation}
so
\[
\mu_{\mathcal{R}}^{t}\left(  \omega,\xi\right)  =\sum_{m\equiv n\text{
}(\text{{\small mod}}\left\vert p_{\omega}\right\vert ^{2})}\mu_{\mathcal{R}%
}^{t}\left(  \omega,\xi\right)  \left(  m,n\right)  \phi_{m}^{\omega}%
\otimes\overline{\phi_{n}^{\omega}}.
\]
Given $k\in\mathbb{Z}^{d}$, let $\omega\in\mathbb{W}$ be the unique resonant
direction such that $k\in\omega$. In view of (\ref{e0}), (\ref{e1}),
(\ref{e2}), we have that, for a.e. $t\in\mathbb{R}$ the $k^{\text{th}}$
Fourier coefficient of $\mu\left(  t,\cdot\right)  $ is given by:
\[
\int_{\mathbb{T}^{d}}\overline{\psi_{k}\left(  x\right)  }\mu\left(
t,dx,\xi\right)  =\frac{1}{\left(  2\pi\right)  ^{d/2}}\sum_{m-n=k\cdot
p_{\omega}}\mu_{\mathcal{R}}^{t}\left(  \omega,\xi\right)  \left(  m,n\right)
.
\]
The trace density of $\mu_{\mathcal{R}}^{t}\left(  \omega,\xi\right)  $ is
precisely the measure%
\[
\tilde{\rho}_{\omega}^{t}\left(  \cdot,\xi\right)  :=\frac{1}{\sqrt
{2\pi\left\vert p_{\omega}\right\vert }}\sum_{k\in\omega\cap\mathbb{Z}^{d}%
}\sum_{m-n=k\cdot p_{\omega}}\mu_{\mathcal{R}}^{t}\left(  \omega,\xi\right)
\left(  m,n\right)  \phi_{k\cdot p_{\omega}}^{\omega}+\frac{1}{2\pi\left\vert
p_{\omega}\right\vert }\operatorname{tr}\mu_{\mathcal{R}}^{t}\left(
\omega,\xi\right)  .
\]
As, for $k\in\omega\cap\mathbb{Z}^{d}$ one has that $\phi_{k\cdot p_{\omega}%
}^{\omega}$ is $L_{\omega}/\left\vert p_{\omega}\right\vert ^{2}$-periodic and%
\[
\mathcal{E}_{\omega}\phi_{k\cdot p_{\omega}}^{\omega}=\sqrt{\frac{\left\vert
p_{\omega}\right\vert }{\left(  2\pi\right)  ^{d-1}}}\psi_{k},
\]
we find that%
\begin{align*}
\frac{1}{\left(  2\pi\right)  ^{d/2}}\sum_{k\in\omega\cap\mathbb{Z}^{d}}%
\sum_{m-n=k\cdot p_{\omega}}\mu_{\mathcal{R}}^{t}\left(  \omega,\xi\right)
\left(  m,n\right)  \psi_{k}  &  =\mathcal{E}_{\omega}\left(  \tilde{\rho
}_{\omega}^{t}\left(  \cdot,\xi\right)  -\operatorname{tr}\mu_{\mathcal{R}%
}^{t}\left(  \omega,\xi\right)  \right) \\
&  =\rho_{\omega}^{t}\left(  \cdot,\xi\right)  ,
\end{align*}
and identity (\ref{MainFormula}) follows.
\end{proof}

The proof of Corollary \ref{CorPosDen} is an easy consequence of formula
(\ref{MainFormula}) and the properties of resonant Wigner
distributions.\smallskip

\begin{proof}
[Proof of Corollary \ref{CorPosDen}]The weak form of Egorov's theorem proved
in \cite{MaciaAvSchrod}, Theorem 2 (ii), gives in this case:%
\begin{equation}
\int_{T^{\ast}\mathbb{T}^{2}}b\left(  \xi\right)  \mu\left(  t,dx,d\xi\right)
=\int_{T^{\ast}\mathbb{T}^{2}}b\left(  \xi\right)  \mu_{0}\left(
dx,d\xi\right)  , \label{Identity}%
\end{equation}
for every $b\in C_{c}\left(  \mathbb{R}^{2}\right)  $ and a.e. $t\in
\mathbb{R}$ (this can also be directly deduced from equation (\ref{Liouville}%
)). Identity (\ref{Identity}), together with our hypothesis $\mu_{0}\left(
\left\{  \xi=0\right\}  \right)  =0$ implies $\mu\left(  t,\left\{
\xi=0\right\}  \right)  =0$ for a.e. $t\in\mathbb{R}$. Notice that since $d=2$
the lines $I_{\omega}$ only intersect at the origin. As a consequence of this,
we deduce that%
\begin{equation}
\mu\left(  t,\cdot\right)  =\sum_{\omega\in\mathbb{W}}\mu\left(
t,\cdot\right)  \rceil_{\mathbb{T}^{2}\times I_{\omega}}+\mu\left(
t,\cdot\right)  \rceil_{\mathbb{T}^{2}\times\left(  \mathbb{R}^{2}%
\setminus\Omega\right)  }. \label{sum}%
\end{equation}
Let $\nu_{0}:=\left(  2\pi\right)  ^{-2}\int_{\mathbb{T}^{2}}\mu_{0}\left(
dy,\cdot\right)  $ , then we obtain from formula (\ref{MainFormula}) the
following expressions:%
\[
\mu\left(  t,\cdot\right)  \rceil_{\mathbb{T}^{2}\times I_{\omega}}%
=\rho_{\omega}^{t}+\nu_{0}\rceil_{I_{\omega}},\quad\mu\left(  t,\cdot\right)
\rceil_{\mathbb{T}^{2}\times\left(  \mathbb{R}^{2}\setminus\Omega\right)
}=\nu_{0}\rceil_{\mathbb{R}^{d}\setminus\Omega}.
\]
Recall that all the measures involved in the right hand side of equation
(\ref{sum}) are mutually disjoint. In particular, the fact that $\mu\left(
t,\cdot\right)  \geq0$ for a.e. $t\in\mathbb{R}$ implies that $\rho_{\omega
}^{t}+\nu_{0}\rceil_{I_{\omega}}\geq0$ as well. Theorem \ref{ThmMain} shows
that the projection on $\mathbb{T}^{2}$ of every $\rho_{\omega}^{t}+\nu
_{0}\rceil_{I_{\omega}}$ is absolutely continuous with respect to the Lebesgue
measure. Therefore, the monotone convergence theorem ensures that
$\int_{\mathbb{R}^{2}}\mu\left(  t,\cdot,d\xi\right)  $ is also absolutely
continuous with respect to the Lebesgue measure for a.e. $t\in\mathbb{R}$. The
result then follows applying identity (\ref{convposden}).
\end{proof}

\section{\label{SecEx}Additional properties and examples}

The following is a direct consequence of Proposition \ref{Prop2mircoCon} and
Theorem \ref{ThmMainComplete}.

\begin{proposition}
\label{PropRhoZero}Let $\left(  u_{h}\right)  $ be an $h$-oscillating sequence
such that (\ref{conv}) holds. If in addition, one has%
\[
\lim_{h\rightarrow0^{+}}\sum_{\left\vert k\cdot p_{\omega}\right\vert
<N}\left\vert \widehat{u_{h}}\left(  k\right)  \right\vert ^{2}=0,\qquad
\text{for every }N>0\text{,}%
\]
for some $\omega\in\mathbb{W}$ then the corresponding term $\rho_{\omega}^{t}$
in (\ref{MainFormula}) vanishes identically.
\end{proposition}

As an example, we shall apply Proposition \ref{PropRhoZero} to analyse the
propagation of wave-packet type solutions to the Schr\"{o}dinger equation in
this context. Given $\left(  x_{0},\xi_{0}\right)  \in T^{\ast}\mathbb{T}^{d}$
and $\rho\in C_{c}^{\infty}(\left(  -\pi,\pi\right)  ^{d})$ define $c_{\left(
x_{0},\xi_{0}\right)  }^{h}\in L^{2}\left(  \mathbb{T}^{d}\right)  $ to be the
$2\pi\mathbb{Z}^{d}$-periodization of the function:%
\[
u_{h}\left(  x\right)  :=\frac{1}{h^{d/4}}\rho\left(  \frac{x-x_{0}}{\sqrt{h}%
}\right)  e^{i\xi_{0}/h\cdot x}.
\]
The Poisson summation formula ensures that the Fourier coefficients of
$c_{\left(  x_{0},\xi_{0}\right)  }^{h}$ are:%
\[
\widehat{c_{\left(  x_{0},\xi_{0}\right)  }^{h}}\left(  k\right)  =\left(
2\pi\right)  ^{d/2}h^{d/4}\widehat{\rho}(\sqrt{h}\left(  k-\xi_{0}/h\right)
)e^{-i\left(  k-\xi_{0}/h\right)  \cdot x_{0}}.
\]
It is not difficult to prove that
\[
w_{c_{\left(  x_{0},\xi_{0}\right)  }^{h}}^{h}\left(  0,\cdot\right)
\rightharpoonup\left\Vert \rho\right\Vert _{L^{2}\left(  \mathbb{R}%
^{d}\right)  }^{2}\delta_{x_{0}}\otimes\delta_{\xi_{0}},\quad\text{as
}h\rightarrow0^{+}.
\]

\begin{proposition}
\label{PropWP}Let $\mu$ be the semiclassical measure given by the limit
(\ref{conv}) corresponding to the initial data $c_{\left(  x_{0},\xi
_{0}\right)  }^{h}$. Then, for almost every $t\in\mathbb{R}$,%
\[
\mu\left(  t,x,\xi\right)  =\left\Vert \rho\right\Vert _{L^{2}\left(
\mathbb{R}^{d}\right)  }^{2}dx\delta_{\xi_{0}}\left(  \xi\right)  .
\]

\end{proposition}

\begin{proof}
Let $\omega\in\mathbb{W}$; if $\xi_{0}\notin I_{\omega}$ then $\rho_{\omega
}^{t}=0$; therefore we may suppose that $\xi_{0}\in I_{\omega}$. The
conclusion will follow as soon as we show that $\rho_{\omega}^{t}=0$ also
holds in this case. Take a function $\chi\in C_{c}^{\infty}\left(  \left(
-2,2\right)  \right)  $, identically equal to one in $\left[  -1,1\right]  $
and taking values between $0$ and $1$. Let $N>0$ and write $\chi_{N}\left(
\xi\right)  :=\chi\left(  p_{\omega}\cdot\xi/N\right)  $. Then%
\[
\sum_{\left\vert k\cdot p_{\omega}\right\vert <N}\left\vert \widehat
{c_{\left(  x_{0},\xi_{0}\right)  }^{h}}\left(  k\right)  \right\vert ^{2}%
\leq\left(  2\pi\right)  ^{d}h^{d/2}\sum_{k\in\mathbb{Z}^{d}}\left\vert
\chi_{N}\left(  k\right)  \widehat{\rho}(\sqrt{h}k-\xi_{0}/\sqrt
{h})\right\vert ^{2}.
\]
Applying the Poisson summation formula, we find that the right hand side of
the above inequality equals $\left\Vert v_{h}\right\Vert _{L^{2}\left(
\mathbb{T}^{d}\right)  }^{2}$, where $v_{h}$ stands for the $2\pi
\mathbb{Z}^{d}$-periodization of $\chi_{N}\left(  D_{x}\right)  u_{h}$. Since
this function is in $\mathcal{S}\left(  \mathbb{R}^{d}\right)  $, we have, for
every $s>d$ an estimate:%
\[
\left\Vert v_{h}\right\Vert _{L^{2}\left(  \mathbb{T}^{d}\right)  }^{2}\leq
C_{s}\int_{\mathbb{R}^{d}}\left\vert \chi_{N}\left(  D_{x}\right)
u_{h}\left(  x\right)  \right\vert ^{2}\left(  1+\left\vert x\right\vert
^{2}\right)  ^{s/2}dx.
\]
Applying Plancherel's identity, we get, after changing variables and taking
into account that $\xi_{0}\cdot p_{\omega}=0$,%
\[
\int_{\mathbb{R}^{d}}\left\vert \chi_{N}\left(  D_{x}\right)  u_{h}\left(
x\right)  \right\vert ^{2}\left(  1+\left\vert x\right\vert ^{2}\right)
^{s/2}dx=\int_{\mathbb{R}^{d}}\left\vert \left(  1-h\Delta_{\xi}\right)
^{s/4}w_{h}\left(  \xi\right)  \right\vert ^{2}\frac{d\xi}{\left(
2\pi\right)  ^{d}},
\]
where $w_{h}\left(  \xi\right)  :=\chi_{N}\left(  \xi/\sqrt{h}\right)
\widehat{\rho}\left(  \xi\right)  $. The functions $\left(  1-h\Delta_{\xi
}\right)  ^{s/4}w_{h}$ are uniformly bounded in $\mathcal{S}\left(
\mathbb{R}^{d}\right)  $, and supported on the strips $S_{h}:=\left\{
\xi:\left\vert \xi\cdot p_{\omega}\right\vert \leq2\sqrt{h}N\right\}  $.
Taking for instance $s/4\in\mathbb{N}$, we find that, for every $\varepsilon
>0$ there exist $R_{\varepsilon}>0$ such that:%
\[
\int_{\mathbb{R}^{d}}\left\vert \left(  1-h\Delta_{\xi}\right)  ^{s/4}%
w_{h}\left(  \xi\right)  \right\vert ^{2}\frac{d\xi}{\left(  2\pi\right)
^{d}}\leq C\left\vert S_{h}\cap B\left(  0;R_{\varepsilon}\right)  \right\vert
+\varepsilon.
\]
This expression tends to $\varepsilon$ as $h\rightarrow0^{+}$. Therefore, as
$\varepsilon$ is arbitrary, Proposition \ref{PropRhoZero} ensures that
$\rho_{\omega}^{t}=0$, and the conclusion follows.
\end{proof}

If instead we consider the purely oscillating profiles $u_{h}\left(  x\right)
$ defined as the periodizations of:%
\[
\rho\left(  x\right)  e^{i\xi_{0}/h\cdot x},
\]
with $\xi_{0}\in\Omega$ a simple resonance (\emph{i.e.} such that $\lambda
\xi_{0}\in\mathbb{Z}^{d}$ for some $\lambda\in\mathbb{R}\setminus\left\{
0\right\}  $) we find that some of the terms $\rho_{\omega}^{t}$ are non-zero.
More precisely, write $\rho_{\text{per}}$ to denote the periodization of
$\rho$ and set, for $f\in L^{1}\left(  \mathbb{T}^{d}\right)  $:%
\[
\left\langle f\right\rangle _{\xi_{0}}\left(  x\right)  :=\frac{\left\vert
\xi_{0}\right\vert }{L}\int_{0}^{L/\left\vert \xi_{0}\right\vert }f\left(
x+s\xi_{0}\right)  ds,
\]
where $L$ denotes the length of the periodic geodesic issued from $\left(
x,\xi_{0}\right)  $. We have the following.

\begin{proposition}
Let $\xi_{0}\in\Omega\setminus\left\{  0\right\}  $ be a simple resonance. The
the semiclassical measure $\mu$ given by (\ref{conv}) corresponding to the
initial data $u_{h_{n}}$, where for simplicity we have taken $\xi_{0}/h_{n}%
\in\mathbb{Z}^{d}$, is given by:%
\[
\mu\left(  t,x,\xi\right)  =\left\langle \left\vert e^{it\Delta_{x}/2}%
\rho_{\text{\emph{per}}}\right\vert ^{2}\right\rangle _{\xi_{0}}\left(
x\right)  dx\delta_{\xi_{0}}\left(  \xi\right)  .
\]

\end{proposition}

\begin{proof}
It is easy to check that the Wigner measure corresponding to the sequence of
initial data $\left(  u_{h_{n}}\right)  $ is precisely:%
\[
\mu_{0}\left(  x,\xi\right)  =\left\vert \rho_{\text{per}}\left(  x\right)
\right\vert ^{2}dx\delta_{\xi_{0}}\left(  \xi\right)  .
\]
Therefore, in view of (\ref{muRltm0}), the resonant Wigner measure of $\left(
u_{h_{n}}\right)  $ satisfies $\mu_{\mathcal{R}}\left(  \omega,\cdot\right)
\equiv0$, whenever $\xi_{0}\notin I_{\omega}$. Let us compute the measures
$\mu_{\mathcal{R}}\left(  \omega,\cdot\right)  $ when $\xi_{0}\in I_{\omega}$.
Start noticing that the Poisson summation formula gives:%
\[
u_{h}=\sum_{k\in\mathbb{Z}^{d}}\widehat{\rho}\left(  k-\xi_{0}/h_{n}\right)
\psi_{k};
\]
hence, for $b\in C^{\infty}\left(  I_{\omega}\right)  $:%
\begin{align*}
&  \int_{I_{\omega}}b\left(  \xi\right)  \mathcal{R}_{u_{h_{n}}}^{h_{n}%
}\left(  \omega,d\xi\right)  \\
&  =\sum_{\left[  c\right]  \in\mathbb{Z}_{\left\vert p_{\omega}\right\vert
^{2}}}\sum_{\substack{m,n\in\left[  c\right]  \\r\in\omega_{c}^{\perp}%
}}\widehat{\rho}\left(  \frac{m}{\left\vert p_{\omega}\right\vert }\nu
_{\omega}+r-\frac{\xi_{0}}{h_{n}}\right)  \overline{\widehat{\rho}\left(
\frac{n}{\left\vert p_{\omega}\right\vert }\nu_{\omega}+r-\frac{\xi_{0}}%
{h_{n}}\right)  }b\left(  hr\right)  \phi_{m}^{\omega}\otimes\overline
{\phi_{n}^{\omega}}.
\end{align*}
Since $\xi_{0}/h_{n}\in I_{\omega}\cap\mathbb{Z}^{d}$, each of the summands in
$m$ and $n$ can be rewritten as:%
\[
\sum_{\substack{k,j\in\mathbb{Z}^{d},k-j=\lambda p_{\omega}\\k\cdot p_{\omega
}=m,j\cdot p_{\omega}=n}}\widehat{\rho}\left(  k\right)  \overline
{\widehat{\rho}\left(  j\right)  }b\left(  hk^{\perp}+\xi_{0}\right)  \phi
_{m}^{\omega}\otimes\overline{\phi_{n}^{\omega}},
\]
for some $\lambda\in\mathbb{Z}$ and where $k^{\perp}$ denotes the projection
of $k$ onto $I_{\omega}$. As $h_{n}\rightarrow0$ this converges to:%
\[
\sum_{\substack{k,j\in\mathbb{Z}^{d},k-j=\lambda p_{\omega}\\k\cdot p_{\omega
}=m,j\cdot p_{\omega}=n}}\widehat{\rho}\left(  k\right)  \overline
{\widehat{\rho}\left(  j\right)  }b\left(  \xi_{0}\right)  \phi_{m}^{\omega
}\otimes\overline{\phi_{n}^{\omega}}.
\]
Therefore,
\[
\mu_{\mathcal{R}}\left(  \omega,\xi\right)  =\sum_{\lambda\in\mathbb{Z}}%
\sum_{\substack{k,j\in\mathbb{Z}^{d},k-j=\lambda p_{\omega}\\k\cdot p_{\omega
}=m,j\cdot p_{\omega}=n}}\widehat{\rho}\left(  k\right)  \overline
{\widehat{\rho}\left(  j\right)  }\phi_{m}^{\omega}\otimes\overline{\phi
_{n}^{\omega}}\delta_{\xi_{0}}\left(  \xi\right)  .
\]
The conclusion then follows using the formula defining $\rho_{\omega}^{t}$.
\end{proof}

\section{Appendix: operator-valued measures}

Let $H$ be a separable Hilbert space, we denote by $\mathcal{L}\left(
H\right)  $, $\mathcal{K}\left(  H\right)  $, and $\mathcal{L}^{1}\left(
H\right)  $ the spaces of bounded, compact and trace-class operators on $H$
respectively. If $A\in\mathcal{L}^{1}\left(  H\right)  $, $\operatorname{tr}A$
denotes the trace of $A$; $\left\Vert A\right\Vert _{\mathcal{L}^{1}\left(
H\right)  }:=\operatorname{tr}\left\vert A\right\vert $ defines a norm on
$\mathcal{L}^{1}\left(  H\right)  $. With this norm, $\mathcal{L}^{1}\left(
H\right)  $ is the dual of $\mathcal{K}\left(  H\right)  $, the duality being
$\operatorname{tr}\left(  AB\right)  $.

When is $X$ a locally compact, $\sigma$-compact, Hausdorff metric space, the
space $\mathcal{M}\left(  X;\mathcal{L}^{1}\left(  H\right)  \right)  $ of
trace-operator-valued Radon measures on $X$ consists of linear operators
$\mu:C_{c}\left(  X\right)  \rightarrow\mathcal{L}^{1}\left(  H\right)  $
bounded in the following sense: given $K\subset X$ compact there exists
$C_{K}>0$ such that for every $\varphi\in C_{c}\left(  K\right)  $,
\[
\left\Vert \left\langle \mu,\varphi\right\rangle \right\Vert _{\mathcal{L}%
^{1}\left(  H\right)  }\leq C_{K}\sup_{x\in K}\left\vert \varphi\left(
x\right)  \right\vert .
\]
Note that $\mathcal{M}\left(  X;\mathcal{L}^{1}\left(  H\right)  \right)  $ is
the dual of $C_{c}\left(  X;\mathcal{K}\left(  H\right)  \right)  $, the space
of compactly supported functions from $X$ into $\mathcal{K}\left(  H\right)  $.

An element $\mu\in\mathcal{M}\left(  X;\mathcal{L}^{1}\left(  H\right)
\right)  $ is positive if for every non-negative $\varphi\in C_{c}\left(
X\right)  $ the operator $\left\langle \mu,\varphi\right\rangle $ is Hermitian
and positive. The set of such positive elements is denoted by $\mathcal{M}%
_{+}\left(  X;\mathcal{L}^{1}\left(  H\right)  \right)  $. Given a positive
measure $\mu$ on defines the scalar valued positive measure $\operatorname{tr}%
\mu$ as%
\[
\left\langle \operatorname{tr}\mu,\varphi\right\rangle :=\operatorname{tr}%
\left\langle \mu,\varphi\right\rangle ,\qquad\text{for }\varphi\in
C_{c}\left(  \mathbb{R}^{d}\right)  .
\]
We refer the reader to \cite{GerardMDM} for a clear presentation of
operator-valued measures, as well as a proof of a Radon-Nykodim theorem in
this context.

When $H=L^{2}\left(  T,\nu\right)  $, where $T$ is locally compact, $\sigma
$-compact, Hausdorff metric space equipped with a Radon measure $\nu$, then
the operators $\left\langle \mu,b\right\rangle $ may be represented by their
integral kernels $k_{b}\in L^{2}\left(  T\times T\right)  $. Thus, $\mu$ can
be viewed as an $L^{2}\left(  T\times T\right)  $-valued measure.

Given $f\in C_{c}\left(  T\times X\right)  $, we denote by $m_{f}\left(
x\right)  $ the operator acting on $L^{2}\left(  T,\nu\right)  $ by
multiplication by $f\left(  \cdot,x\right)  $. Clearly, $m_{f}\in C_{c}\left(
X;\mathcal{L}\left(  L^{2}\left(  T\right)  \right)  \right)  $. The following
construction is used in the proof of Theorems \ref{ThmMain} and
\ref{ThmMainComplete}.

\begin{proposition}
\label{PropAbsC}Let $H=L^{2}\left(  T,\nu\right)  $ and $\mu\in\mathcal{M}%
_{+}\left(  X;\mathcal{L}^{1}\left(  H\right)  \right)  $. The linear
operator
\[
\rho_{\mu}:C_{c}\left(  T\times X\right)  \rightarrow\mathbb{R:\;}%
g\mapsto\operatorname{tr}\left\langle \mu,m_{g}\right\rangle
\]
extends to a positive Radon measure on $T\times X$, which is finite if $\mu$
is. Moreover, for every $\varphi\in C_{c}\left(  X\right)  $,%
\begin{equation}
\int_{X}\varphi\left(  x\right)  \rho_{\mu}\left(  \cdot,dx\right)  \in
L^{1}\left(  T,\nu\right)  . \label{Appabscon}%
\end{equation}
In addition
\[
\rho_{\mu}=0\quad\text{if and only if}\quad\mu=0.
\]

\end{proposition}

\begin{proof}
Given compact sets $S\subset T$ and $K\subset X$ and functions $f\in
C_{c}\left(  S\right)  $ and $\varphi\in C_{c}\left(  K\right)  $ we have, as
$\mathcal{L}^{1}\left(  H\right)  $ is an ideal in $\mathcal{L}\left(
H\right)  $:%
\[
\left\vert \operatorname{tr}\left(  m_{f}\left\langle \mu,\varphi\right\rangle
\right)  \right\vert \leq\left\Vert m_{f}\right\Vert _{\mathcal{L}\left(
H\right)  }\operatorname{tr}\left\vert \left\langle \mu,\varphi\right\rangle
\right\vert \leq C_{K}\left\Vert f\right\Vert _{L^{\infty}\left(
T,\nu\right)  }\sup_{x\in K}\left\vert \varphi\left(  x\right)  \right\vert
\]
for some constant $C_{K}>0$. Therefore, as $C_{c}\left(  T\right)  \otimes
C_{c}\left(  X\right)  $ is dense in $C_{c}\left(  T\times X\right)  $, the
functional $\rho_{\mu}$ extends to a measure on $T\times X$. The positivity of
$\rho_{\mu}$ follows easily from the properties of the trace of linear
operators. Let $\varphi\in C_{c}\left(  X\right)  $ and write%
\[
\left\langle \mu,\varphi\right\rangle =\sum_{j=1}^{\infty}\lambda_{j}\phi
_{j}\otimes\overline{\phi_{j}}%
\]
where $\lambda_{j}$ are the eigenvalues of $\left\langle \mu,\varphi
\right\rangle $, the $\phi_{j}$, $j\in\mathbb{N}$, form an orthonormal basis
of $H$ consisting of eigenvectors, and $\phi_{j}\otimes\overline{\phi_{j}}$ is
the projection on the linear span of $\phi_{j}$. Now,%
\begin{equation}
\int_{T\times X}f\left(  t\right)  \varphi\left(  x\right)  \rho_{\mu}\left(
dt,dx\right)  =\operatorname{tr}\left(  m_{f}\sum_{j=1}^{\infty}\lambda
_{j}\phi_{j}\otimes\overline{\phi_{j}}\right)  =\sum_{j=0}^{\infty}\lambda
_{j}\int_{T}f\left\vert \phi_{j}\right\vert ^{2}d\nu. \label{defrhomu}%
\end{equation}
Since $\left\langle \mu,\varphi\right\rangle $ is trace class, $\sum
_{j=0}^{\infty}\lambda_{j}$ is absolutely convergent; therefore,
\[
\int_{X}\varphi\left(  x\right)  \rho_{\mu}\left(  \cdot,dx\right)
=\sum_{j=0}^{\infty}\lambda_{j}\left\vert \phi_{j}\right\vert ^{2}%
\]
is in $L^{1}\left(  T\right)  $ as we wanted to show. Finally, to see that
$\rho_{\mu}=0$ implies $\mu=0$ simply take $f$ and $\varphi$ non-negative in
formula (\ref{defrhomu}). As the left hand side is zero, all $\lambda_{j}$
must vanish, that is $\left\langle \mu,\varphi\right\rangle =0$. Since
$\varphi$ is arbitrary, we must have $\mu=0$.
\end{proof}

\bigskip

\noindent\textbf{Acknowledgments.} Much of this research has been done as the
author was visiting the D\'{e}partement de Math\'{e}matiques at Universit\'{e}
de Paris-Sud and the Mathematics Department at University of Texas at Austin.
He wishes to thank these institutions for their kind hospitality. He also
wants to express his gratitude to Patrick G\'{e}rard for his interest and the
numerous fruitful discussions they had regarding this work.

\end{document}